\documentclass[12pt]{amsart}
\usepackage{a4wide,amssymb,amscd,amsmath,mathtools,stmaryrd,url}

\begin{document}

\newcommand{\genG}{\mathcal{G}}                               
\newcommand{\dirac}{\delta}                                   
\newcommand{\comfam}{\mathcal{K}}                             
\newcommand{\fixco}{\K_0^S}                                   
\newcommand{\tor}{\operatorname{tor}}                         
\newcommand{\PB}{polynomially bounded}
\newcommand{\divides}{\,\big|\,}
\newcommand{\OI}{\mathcal{OI}}
\newcommand{\lcm}{\operatorname{lcm}}
\newcommand{\Hecke}{\mathcal{H}}
\newcommand{\types}{\mathcal{F}}
\newcommand{\TWN}{(TWN)}
\newcommand{\BD}{(BD)}
\newcommand{\C}{{\mathbb C}}
\newcommand{\N}{{\mathbb N}}
\newcommand{\R}{{\mathbb R}}
\newcommand{\Z}{{\mathbb Z}}
\newcommand{\Q}{{\mathbb Q}}
\newcommand{\Id}{\operatorname{Id}}
\newcommand{\A}{{\mathbb A}}
\newcommand{\AF}{{\mathcal A}}
\newcommand{\K}{\mathbf{K}}
\newcommand{\plnch}{\operatorname{pl}}
\newcommand{\bs}{\backslash}
\newcommand{\temp}{\operatorname{temp}}
\newcommand{\disc}{\operatorname{disc}}
\newcommand{\cusp}{\operatorname{cusp}}
\newcommand{\spec}{\operatorname{spec}}
\newcommand{\sprod}[2]{\left\langle#1,#2\right\rangle}
\renewcommand{\Im}{\operatorname{Im}}
\renewcommand{\Re}{\operatorname{Re}}
\newcommand{\Ind}{\operatorname{Ind}}
\newcommand{\tr}{\operatorname{tr}}
\newcommand{\Ad}{\operatorname{Ad}}
\newcommand{\Lie}{\operatorname{Lie}}
\newcommand{\Hom}{\operatorname{Hom}}
\newcommand{\Ker}{\operatorname{Ker}}
\newcommand{\vol}{\operatorname{vol}}
\newcommand{\SL}{\operatorname{SL}}
\newcommand{\GL}{\operatorname{GL}}
\newcommand{\card}[1]{\lvert#1\rvert}
\newcommand{\abs}[1]{\lvert{#1}\rvert}
\newcommand{\norm}[1]{\lVert#1\rVert}
\newcommand{\one}{\mathbf 1}
\newcommand{\aaa}{\mathfrak{a}}
\newcommand{\eps}{\epsilon}
\newcommand{\ad}{\operatorname{ad}}
\newcommand{\proj}{\operatorname{proj}}
\newcommand{\rest}{\big|}
\newcommand{\dsum}{\oplus}
\newcommand{\univ}{\mathcal{U}}
\newcommand{\modulus}{\delta}
\newcommand{\AAA}{A}
\renewcommand{\Re}{\operatorname{Re}}
\renewcommand{\Im}{\operatorname{Im}}
\newcommand{\inorm}{\operatorname{N}}
\newcommand{\nnn}{\mathfrak{n}}
\newcommand{\fin}{{\operatorname{fin}}}
\newcommand{\level}{\operatorname{lev}}
\newcommand{\minlevel}{\operatorname{minlev}}
\newcommand{\fctr}[1]{G_{#1}}                
\newcommand{\der}{\operatorname{der}}
\newcommand{\cls}{\mathfrak{o}}
\newcommand{\truncJ}{\tilde J}
\newcommand{\SC}{\operatorname{sc}}
\newcommand{\fctrz}{\mathfrak{F}}
\newcommand{\idl}{\mathfrak{I}}
\newcommand{\intgr}{\mathcal{O}}
\newcommand{\Nm}{\operatorname{Nm}}
\newcommand{\primes}{\mathcal{P}}
\newcommand{\pr}{\operatorname{pr}}
\newcommand{\alfctrs}[1]{\widetilde{#1}}
\newcommand{\cntgr}{\Theta}

\newcommand{\sm}[4]{\left(\begin{smallmatrix}{#1}&{#2}\\{#3}&{#4}\end{smallmatrix}\right)}

\newcommand{\Lieg}{\mathfrak{g}}
\newcommand{\Lieh}{\mathfrak{h}}
\newcommand{\Pro}{\operatorname{Pr}}
\newcommand{\nc}{\operatorname{nc}}
\newcommand{\Res}{\operatorname{Res}}

\newtheorem{theorem}{Theorem}[section]
\newtheorem{lemma}[theorem]{Lemma}
\newtheorem{proposition}[theorem]{Proposition}
\newtheorem{remark}[theorem]{Remark}
\newtheorem{conjecture}[theorem]{Conjecture}
\newtheorem{definition}[theorem]{Definition}
\newtheorem{corollary}[theorem]{Corollary}
\newtheorem{example}[theorem]{Example}
\newtheorem{claim}[theorem]{``Claim''}
\newtheorem{assumption}[theorem]{Assumption}
\newtheorem{wassumption}[theorem]{Working Assumption}

\newenvironment{thmbis}[1]
  {\renewcommand{\thetheorem}{\ref{#1}$'$}%
   \addtocounter{theorem}{-1}%
   \begin{theorem}}
  {\end{theorem}}

\numberwithin{equation}{section}

\title[Limit multiplicity for congruence subgroups]{An approximation principle for congruence subgroups II: application to the limit multiplicity problem}
\author{Tobias Finis}
\address{Universit\"at Leipzig, Mathematisches Institut, Postfach 100920, D-04009 Leipzig, Germany}
\thanks{T.F. partially supported by DFG Heisenberg grant \# FI 1795/1-1.}
\email{finis@math.uni-leipzig.de}
\author{Erez Lapid}
\address{Department of Mathematics, The Weizmann Institute of Science, Rehovot 76100, Israel}
\thanks{E.L. partially supported by grant \#711733 from the Minerva Foundation}
\email{erez.m.lapid@gmail.com}
\date{\today}



\maketitle

\setcounter{tocdepth}{1}
\tableofcontents

\section{Introduction and statement of results}
The limit multiplicity problem in its classical formulation (introduced by DeGeorge--Wallach \cite{MR0492077, MR534759, MR562677})
concerns the asymptotic behavior of the spectra of lattices in semisimple Lie groups.
For a locally compact group $\genG$ (with a fixed choice of a Haar measure) let
$\Pi (\genG)$ be its unitary dual.
Define the discrete spectral measure on $\Pi (\genG)$ associated to a lattice $\Gamma$ in $\genG$ by
\[
\mu_\Gamma=\frac1{\vol(\Gamma\bs \genG)}\sum_{\pi\in\Pi(\genG)} \dim\Hom_{\genG} (\pi, L^2 (\Gamma \bs \genG)) \dirac_\pi,
\]
where $\dirac_\pi$ denotes the Dirac measure at $\pi$.
(The multiplicities in this definition will be finite in the cases of interest to us.)
Let now $\genG$ be more specifically a connected linear semisimple Lie group,
and denote the Plancherel measure of $\genG$ by $\mu_{\plnch}$.
We ask: under which conditions does
a sequence $(\Gamma_n)$ of lattices in $\genG$ with $\vol (\Gamma_n \bs \genG) \to \infty$
satisfy the \emph{limit multiplicity property} $\mu_{\Gamma_n} \to \mu_{\plnch}$?
In this formulation it is natural to impose that $\Gamma_n\cap Z(\genG)=1$
for almost all $n$, where $Z (\genG)$ is the (finite) center of $\genG$.
Slightly more generally, for an arbitrary subgroup $\cntgr$ of $Z (\genG)$
we can ask whether $\mu_{\Gamma_n} \to \mu_{\plnch,\cntgr}$, the Plancherel measure of $\genG / \cntgr$, provided that
$\Gamma_n \cap Z (\genG) = \cntgr$ for almost all $n$.
Here, the convergence is to be understood in the sense that $\mu_{\Gamma_n} (A) \to \mu_{\plnch,\cntgr} (A)$ for suitable subsets $A \subset \Pi (\genG)$ (namely
Jordan measurable subsets of the tempered dual and bounded subsets of the non-tempered dual, cf.
Definition \ref{DefLimitMult} below).

We expect this property to hold in great generality, namely
for arbitrary sequences of irreducible congruence arithmetic lattices (i.e., lattices $\Gamma$ containing a principal congruence subgroup
of a standard arithmetic lattice).
We refer to
\cite[\S 1]{FLM3} for a more extensive introduction to this problem (including references to previous work).
While quite general results have been recently obtained in \cite{ABBGNRS} for the case of \emph{uniform} lattices in higher-rank Lie groups,
the results in the \emph{non-uniform} case are not as complete, especially for general groups $\genG$.
In \cite{FLM3}, the collection of the \emph{principal} congruence subgroups of a fixed standard arithmetic lattice $\Gamma = G (\intgr_F)$
was considered, where $G$ is a reductive group over a number field $F$ and $\intgr_F$ is the ring of integers of $F$,
and the limit multiplicity problem was solved affirmatively for this collection under certain natural conditions on $G$ (called
properties \TWN\ and \BD\ in [ibid.]). The goal of this paper is to extend the results of [ibid.] to the collection of \emph{all} congruence subgroups of such a lattice.
The main new ingredient is a purely group-theoretic estimate from \cite{1308.3604} (see below).
Our proof of this estimate in [ibid.] is based on the approximation principle for congruence subgroups introduced in this paper.
(See \cite[\S5]{ABBGNRS} for a different approach.)

In the following we will work in the adelic setting. See Corollary \ref{CorClassical} below for a restatement of the main result in
classical language.

Throughout let $G$ be a (connected) reductive group defined over a number field $F$ and $S$ a finite set of places of $F$ containing the
set $S_\infty$ of all archimedean places. We write $S_{\fin}=S-S_\infty$. Let $F_S$ be the product over all $v\in S$ of the completions $F_v$, $\A^S$ the
restricted product of the $F_v$ for $v \notin S$, and $\A = F_S\times \A^S$ the ring of adeles of $F$.
As usual, $G(F_S)^1$ denotes the intersection of the kernels of the homomorphisms
$\abs{\chi}_S: G (F_S) \to \R^{>0}$, where $\chi$ ranges over the $F$-rational characters of $G$ and $\abs{\cdot}_S$
denotes the normalized absolute value on $F^\times_S$. The subgroup $G (\A)^1$ of $G(\A)$ is defined analogously
using the adelic absolute value $\abs{\cdot}_{\A}$ on $\A^\times$.
Fix a Haar measure on $G(\A)^1$ and on $G(\A^S)$. This determines a Haar measure on $G(F_S)^1$.

For any open compact subgroup $K$ of $G(\A^S)$ let $\mu_K = \mu^{G,S}_K$ be the measure on
the unitary dual $\Pi(G(F_S)^1)$ of $G(F_S)^1$ given by
\begin{align*}
\mu_K & = \frac{1}{\vol(G(F)\bs G(\A)^1/K)} \sum_{\pi\in\Pi(G(F_S)^1)}\dim\Hom_{G(F_S)^1}(\pi,L^2(G(F)\bs G(\A)^1/K)) \, \dirac_\pi \\
& = \frac{\vol(K)}{\vol(G(F)\bs G(\A)^1)} \sum_{\pi\in\Pi (G(\A)^1)} \dim\Hom_{G(\A)^1}(\pi,L^2(G(F)\bs G(\A)^1)) \, \dim (\pi^S)^K \,  \dirac_{\pi_S}.
\end{align*}

Let $Z=Z_G$ be the center of $G$ and $Z(\intgr_S)$ the intersection of $Z(F)$ with the unique maximal compact subgroup of $Z(\A^S)$.
It is well-known that $Z(\intgr_S)$ is a finitely generated abelian group which is a uniform lattice in $Z(F_S)^1 = Z (F_S) \cap G (F_S)^1$.
For any compact open subgroup $K$ of $G(\A^S)$ let
\[
Z_K= Z(F) \cap K \subset Z(\intgr_S).
\]
We consider the set of subgroups of $Z(\intgr_S)$ as a closed subset of $\{0,1\}^{Z(\intgr_S)}$
(which is a compact topological space when endowed with the product topology).
Since any subgroup of $Z(\intgr_S)$ is finitely generated, we have $Z_n\rightarrow\cntgr$ precisely when $\cntgr\subset Z_n$ for all but finitely
many $n$, while each element of the complement of $\cntgr$ in $Z(\intgr_S)$ belongs to only finitely many of the groups $Z_n$.
Thus, $Z_{K_n}\rightarrow\cntgr$ if and only if
\begin{equation} \label{eq: chabauty}
\sum_{z\in Z(F)}(h\otimes\one_{K_n})(z)\rightarrow\sum_{z\in\cntgr}h(z)
\end{equation}
(in either the discrete topology on $\C$ or the usual one) for any $h\in C_c^\infty(Z(F_S)^1)$.
We also remark that if $Z_{K_n}\rightarrow\cntgr$ then $\mu_{K_n}$ is supported on $\Pi(G(F_S)^1/\cntgr)$ for all but finitely many $n$.

\begin{remark}
Using Chevalley's theorem \cite{MR0044570}, one can show that any subgroup $\cntgr$ of $Z (\intgr_S)$
is a limit of subgroups of the form $Z_ {K_n}$ for suitable compact open subgroups $K_n$ of $G(\A^S)$.
We will not go into details, since we will not use this fact.
\end{remark}

\begin{definition} \label{DefLimitMult}
Suppose that $\comfam$ is a set of compact open subgroups of $G(\A^S)$.
We say that $\comfam$ has the \emph{limit multiplicity property} if for any subgroup $\cntgr$ of $Z(\intgr_S)$ and any subset
$\comfam'\subset\comfam$ such that $Z_K\rightarrow\cntgr$, $K\in\comfam'$, we have
$\mu_K\rightarrow\mu_{\plnch,\cntgr}$, $K\in\comfam'$,
where $\mu_{\plnch,\cntgr}$ is the Plancherel measure of $G(F_S)^1/\cntgr$, in the sense that
\begin{enumerate}
\item for any Jordan measurable subset $A\subset\Pi_{\temp}(G(F_S)^1/\cntgr)$ we have
$\mu_K(A)\rightarrow\mu_{\plnch,\cntgr}(A)$, $K\in\comfam'$, and,
\item for any bounded subset $A\subset\Pi(G(F_S)^1)\setminus\Pi_{\temp}(G(F_S)^1)$ we have
$\mu_K(A)\rightarrow0$, $K\in\comfam'$.
\end{enumerate}
\end{definition}

Here, $\Pi_{\temp}(G(F_S)^1/\cntgr) \subset \Pi (G(F_S)^1/\cntgr)$ denotes the tempered dual of 
$G(F_S)^1/\cntgr$ (the support of the Plancherel measure $\mu_{\plnch,\cntgr}$).  
We can rephrase the first condition by saying that for any Riemann integrable function $f$
on $\Pi_{\temp}(G(F_S)^1/\cntgr)$ we have
\[
\mu_K(f)\rightarrow\mu_{\plnch,\cntgr}(f), \quad K\in\comfam'.
\]
Recall that a Jordan measurable subset $A$ of $\Pi_{\temp} (G(F_S)^1/\cntgr)$ is a bounded set such that $\mu_{\plnch,\cntgr}(\partial A)=0$,
where $\partial A=\bar A-A^\circ$ is the boundary of $A$ in $\Pi_{\temp} (G(F_S)^1/\cntgr)$.
A Riemann integrable function on $\Pi_{\temp} (G(F_S)^1/\cntgr)$ is a bounded, compactly supported function
which is continuous almost everywhere with respect to the Plancherel measure.

If $G$ is $F$-simple and simply connected, then we expect the limit multiplicity property to hold for any collection $\comfam$ of
compact open subgroups of $G (\A^S)$ with $\vol (K) \to 0$, $K \in \comfam$. In the general case it is natural to impose the following condition.
For any reductive group $H$ let
$H(\A)^+$ be the image of the map $H^{\SC}(\A)\rightarrow H(\A)$, where $H^{\SC}$ is the simply connected
cover of the derived group of $H$. Define $H(\A^S)^+$ analogously.

\begin{definition} \label{DefinitionNondeg}
We say that a family $\comfam$ of compact open subgroups of $G (\A^S)$ is \emph{non-degenerate},
if for any $F$-simple normal subgroup $H$ of $G$
we have $\vol_{H(\A^S)^+} (K\cap H(\A^S)^+) \to 0$, $K \in \comfam$.
\end{definition}

In this paper we will only treat the case where the family $\comfam$ consists of open subgroups of a fixed open compact subgroup $\fixco$ of $G (\A^S)$.
In this case, non-degeneracy simply amounts to the condition that for any $F$-simple normal subgroup $H$ of $G$ the
map $K\in\comfam\mapsto K\cap H(\A^S)^+$ is finite-to-one. When $G$ itself is $F$-simple and simply connected, we may take
$\comfam$ to be the collection of all open subgroups of $\fixco$.

Recall that in \cite[\S 5]{FLM3}, 
two natural properties of the reductive group $G$ (called
properties \TWN\ and \BD) were introduced and studied.
They concern the behavior of the intertwining operators associated to proper parabolic subgroups of $G$,
and are therefore trivially satisfied if $G$ is anisotropic modulo the center.
By [ibid., Proposition 5.5, Theorem 5.15], they are known to hold if $G$ is either $\GL(n)$ or $\SL(n)$.
In a future paper we will establish them in many additional cases.

Our main result here is the following.
\begin{theorem} \label{MainTheorem}
Suppose that $G$ satisfies \TWN\ and \BD\ and let $\fixco$ be a compact open subgroup of $G (\A^S)$.
Then limit multiplicity holds for any non-degenerate family $\comfam$ of open subgroups of $\fixco$.
\end{theorem}

As a direct consequence we obtain the following result on the spectra of arithmetic lattices.

\begin{corollary} \label{CorClassical}
Suppose that $G$ is a simply connected $F$-simple group satisfying \TWN\ and \BD\ and that $G (F_S)$ is not compact.
Let $\fixco$ be a compact open subgroup of $G (\A^S)$.
Then limit multiplicity holds for the lattices $\Gamma_K = G (F) \cap K \subset G (F_S)$, where
$K$ ranges over the open subgroups of $\fixco$, i.e., for any subgroup $\cntgr$ of the finite group $Z (\intgr_S) \subset G (F_S)$
and any collection $\comfam$ of open subgroups of $\fixco$ such that $Z_K = \Gamma_K \cap Z (\intgr_S)=\cntgr$
for all but finitely many $K \in \comfam$,
we have $\mu_{\Gamma_K} \to \mu_{\plnch,\cntgr}$ (in the sense of Definition \ref{DefLimitMult}).
\end{corollary}

The corollary follows directly from the strong approximation theorem \cite[Theorem 7.12]{MR1278263}, 
which asserts that (under the above conditions on $G$) 
the $G(F_S)$-spaces $G (F) \bs G (\A) / K$ and $\Gamma_K \bs G (F_S)$ 
are canonically isomorphic for any open subgroup $K$ of $G(\A^S)$.
The finite index subgroups $\Gamma_K = G(F) \cap K$ of the arithmetic lattice
$\Gamma_0 = G (F) \cap \fixco$ are called the \emph{congruence subgroups} of $\Gamma_0$.
Since the groups $\SL(n)$ are known to satisfy properties \TWN\ and \BD, we have in particular the
following result.

\begin{corollary}
For any number field $F$ and any finite set $S \supset S_\infty$,
limit multiplicity (in the sense of Definition \ref{DefLimitMult}) holds for the family of
all congruence subgroups $\Gamma_K$ of the lattice $\Gamma_0 = \SL (n, \intgr_S)$ in the group $\SL (n, F_S)$.
\end{corollary}

It is known from the work of Raghunathan \cite{MR0507030,MR842049}, that any isotropic
$F$-simple simply connected group $G$
for which the sum over $v \in S$ of the $F_v$-ranks of $G$ is at least two,
has the \emph{congruence subgroup property}. This means that every finite index subgroup of $\Gamma_0$
is contained in a congruence subgroup whose index is bounded in terms of $G$ only.
In fact, in many cases every finite index subgroup is a congruence subgroup,
for instance if either $S_{\fin}\ne\emptyset$ or if $G$ is split and $F$ is not totally complex.
In such a situation, Corollary \ref{CorClassical} becomes
a statement on the collection of \emph{all} finite index subgroups of the lattice $\Gamma_0$. We will comment on possible extensions of Theorem \ref{MainTheorem} and
Corollary \ref{CorClassical} at the end of \S \ref{SectionStrategy} below.

We thank the Hausdorff Research Institute for Mathematics, Bonn, where a part of this paper was worked out.
We thank Nicolas Bergeron, Tsachik Gelander, Peter Sarnak and Andreas Thom
for useful discussions and interest in the subject matter of this paper.

\section{The proof strategy} \label{SectionStrategy}
We now explain the strategy of the proof of Theorem \ref{MainTheorem}.
Using (a slight extension of) a result of Sauvageot \cite{MR1468833}, we can interpret the limit multiplicity property
in terms of the trace formula as follows.
Fix a maximal compact subgroup $\K_S$ of $G (F_S)$ and
let $\Hecke(G(F_S)^1)$ be the algebra of smooth, compactly supported, bi-$\K_S$-finite functions on $G(F_S)^1$.
For any $h\in \Hecke (G(F_S)^1)$ let $\hat h$ be the function on $\Pi(G(F_S)^1)$ given by $\hat h(\pi)=\tr\pi(h)$.
Denote by $R_{\disc}$ the regular representation of $G(\A)^1$ on the discrete part of $L^2(G(F)\bs G(\A)^1)$.
Note that we have
\[
\mu_K (\hat{h}) = \frac{1}{\vol(G(F)\bs G(\A)^1)} \tr R_{\disc}(h\otimes\one_{K})
\]
and
\[
\mu_{\plnch,\cntgr} (\hat{h}) = \sum_{z\in\cntgr}h (z).
\]
Then we have the following reduction.

\begin{theorem} \label{thm: reduction}
Suppose that a collection $\comfam$ of compact open subgroups of $G (\A^S)$
has the property that for any function $h\in\Hecke(G(F_S)^1)$ we have
\begin{equation} \label{eq: main1}
\mu_K (\hat{h})-\sum_{z\in Z(F) \cap K}h(z)\rightarrow0, \quad K \in \comfam.
\end{equation}
Then limit multiplicity holds for $\comfam$.
\end{theorem}

\begin{proof}
Using \eqref{eq: chabauty}, for any subgroup $\cntgr \subset Z (\intgr_S)$ and any collection
$\comfam' \subset \comfam$ with $Z_K \to \cntgr$, $K \in \comfam'$, we have
$\mu_K (\hat{h}) \to \mu_{\plnch,\cntgr} (\hat{h}) = \sum_{z\in\cntgr}h (z)$, $K \in \comfam'$.
We need to show that this implies the limit multiplicity property in the sense of Definition \ref{DefLimitMult}, i.e.,
the convergence $\mu_K (A) \to \mu_{\plnch,\cntgr} (A)$, $K \in \comfam'$, for suitable sets $A \subset \Pi (G (F_S)^1)$.

Fix a subgroup $\cntgr$ as above.
Arguing as in \cite[\S 2]{FLM3}, the assertion follows from the following variant of Sauvageot's results.

Let $\epsilon>0$.
\begin{enumerate}
\item For any bounded set $A\subset\Pi(G(F_S)^1)\setminus\Pi_{\temp}(G(F_S)^1)$
there exists $h\in\Hecke(G(F_S)^1)$ such that
\begin{enumerate}
\item $\hat h(\pi)\ge0$ for all $\pi\in\Pi(G(F_S)^1)$,
\item $\hat h(\pi)\ge1$ for all $\pi\in A$,
\item $\mu_{\plnch,\cntgr}(\hat h)<\epsilon$.
\end{enumerate}
\item For any Riemann integrable function $f$ on $\Pi_{\temp}(G(F_S)^1/\cntgr)$
there exist $h_1,h_2\in\Hecke(G(F_S)^1)$ such that
\begin{enumerate}
\item $\abs{f(\pi)-\hat h_1(\pi)}\le\hat h_2(\pi)$ for all $\pi\in\Pi(G(F_S)^1/\cntgr)$,
where we extend $f$ by zero to $\Pi (G(F_S)^1/\cntgr)$.
\item $\mu_{\plnch,\cntgr}(\hat{h}_2) <\epsilon$.
\end{enumerate}
\end{enumerate}

The case $\cntgr=1$ is \cite[Theorem 2.1]{FLM3}, due to Sauvageot.
The case where $\cntgr$ is finite easily follows from the fact that if $\hat h (\pi) \ge0$
for all $\pi \in \Pi (G (F_S)^1)$
then $\abs{h(z)}\le h(1)$ for all $z\in Z(F_S)^1$, since
$h(z)=\int_{\Pi_{\temp}(G(F_S)^1)}\hat h(\pi)\omega_\pi(z^{-1}) \, d \mu_{\plnch}$,
where $\omega_\pi$ denotes the central character of $\pi$.

Consider the case of a general subgroup $\cntgr \subset Z (\intgr_S)$.
Let $\cntgr_{\tor}$ be the torsion part of $\cntgr$ and $\overline{\cntgr}=\cntgr/\cntgr_{\tor}$, which is free abelian of finite rank.
Let $X(\genG)$ be the group of all unitary characters of a locally compact group $\genG$ and
consider the restriction map $r: X (G(F_S)^1) \to X (\cntgr)$. 
Since its cokernel is finite, its image 
contains the divisible subgroup $X(\overline{\cntgr}) \subset X (\cntgr)$.
Let $s:X (\overline{\cntgr})\rightarrow X (G(F_S)^1)$ be a (set-theoretic) cross section for $r$ which is almost everywhere continuous.

We start with the proof of the second assertion.
Given a Riemann integrable function $f$ on $\Pi_{\temp}(G(F_S)^1 /\cntgr)$, we define a
Riemann integrable function $F$ on $\Pi_{\temp}(G(F_S)^1 /\cntgr_{\tor})$ by setting
$F(\pi)=f(\pi\otimes s(\omega_\pi|_{\cntgr})^{-1})$.
By the previous case there exist $H_1,H_2 \in\Hecke(G(F_S)^1)$ satisfying the second assertion with $f$ replaced by $F$ and
$\cntgr$ replaced by $\cntgr_{\tor}$.
Let $Y$ be a finite subgroup of $X (\overline{\cntgr})$. Then
the averages $h_j = (\abs{Y}^{-1} \sum_{\chi\in Y} s(\chi)) H_j$ satisfy $\abs{f(\pi)-\hat h_1(\pi)}\le\hat h_2(\pi)$
for all $\pi\in\Pi(G(F_S)^1/\cntgr)$.
Since the support of $H_2$ on $\cntgr$ is finite, we can choose $Y$ such that
$h_2 (z) = ( \abs{Y}^{-1} \sum_{\chi\in Y} \chi (z)) H_2 (z) = 0$ for all $z \in \cntgr$, $z \notin \cntgr_{\tor}$.
Therefore $\mu_{\plnch,\cntgr} (\hat h_2)= \mu_{\plnch,\cntgr_{\tor}} (\hat H_2) < \epsilon$, as desired.

For the first assertion, we again apply the result in the finite group case to obtain a function $H\in\Hecke(G(F_S)^1)$ satisfying the first assertion with
$A$ replaced by $\{\pi\otimes s(\chi):\pi\in A, \chi\in X (\overline{\cntgr})\}$ and $\cntgr$ replaced by $\cntgr_{\tor}$.
Taking $h=(\abs{Y}^{-1} \sum_{\chi\in Y} s(\chi)) H$ as before will yield the result.
\end{proof}

As in \cite{FLM3} we will use Arthur's non-invariant trace formula to attack \eqref{eq: main1}.
Recall that Arthur has defined a certain distribution $h\mapsto J(h)$ on $C_c^\infty(G(\A)^1)$ and expanded it geometrically and spectrally
\cite{MR518111, MR558260, MR681737, MR681738, MR828844, MR835041} (cf. \S\ref{SectionGeom} for more details).
The distribution $J$ depends on the choice of a maximal $F$-split torus of $G$ and
a suitable maximal compact subgroup $\K$ of $G(\A)$.
The main terms on the geometric side are the elliptic orbital integrals, most notably the contribution
$J_{Z(F)}(h)=\vol(G(F)\bs G(\A)^1)\sum_{z\in Z(F)}h(z)$ of the central elements.
The main term on the spectral side is $\tr R_{\disc}(h)$.

In order to prove the relation \eqref{eq: main1} we will consider the following two statements (which together clearly imply it):
\begin{equation} \label{eq: mainspectral}
\text{For any }h\in\Hecke(G(F_S)^1)\text{ we have } J (h \otimes \one_K) - \tr R_{\disc} (h \otimes \one_K) \rightarrow0,\ K\in\comfam,
\end{equation}
and,
\begin{equation} \label{eq: maingeometric}
\text{for any }h\in\Hecke(G(F_S)^1)\text{ we have } J (h \otimes \one_K) -J_{Z(F)} (h \otimes \one_K)\rightarrow 0,\ K\in\comfam.
\end{equation}
Following \cite{FLM3}, we call these relations the \emph{spectral} and \emph{geometric limit properties}, respectively.

In \S\ref{SectionGeom} we will prove the geometric limit property for any non-degenerate family $\comfam$
of open subgroups of $\fixco$ (see Theorem \ref{thm: maingeom} below).
Under assumptions \TWN\ and \BD\ on $G$, we will then prove in \S\ref{sec: spectral limit property}
the spectral limit property for any non-degenerate family $\comfam$ of open subgroups of $\fixco$
(see Corollary \ref{corspectrallimit}).
It is a technical feature of the proof that we first prove for each proper Levi subgroup
$M$ of $G$ a different statement (polynomial
boundedness) on the measures $\mu^{M,S_\infty}_{K_M}$, where 
$K_M$ varies over the open subgroups of an arbitrary compact open subgroup
of $M (\A_{\fin})$. The proof of this auxiliary statement proceeds by induction on the semisimple rank of the Levi subgroup $M$ and 
uses Theorem \ref{thm: maingeom} (for $M$ instead of $G$). 

Both the geometric and spectral limit properties are proved in a quantitative form, i.e., we obtain estimates of the
form $O_{h} (N^{-\delta})$ for the left-hand sides of \eqref{eq: mainspectral} and \eqref{eq: maingeometric},
where $N$ is the appropriately defined level of $K$
and $\delta > 0$ depends only on $G$
(cf. \eqref{EquationMinlevel} below for the precise definition of $N$, which coincides with the standard one if $G$ is $F$-simple and simply connected).
A natural problem, which will not be considered here, is to obtain from this an estimate for the difference $\abs{\mu^{G, S}_{K} (A) - \mu_{\plnch,\cntgr} (A)}$
in terms of $N$ for suitable subsets $A \subset \Pi (G (F_S)^1)$. This would require a quantitative version
of the density principle (Theorem \ref{thm: reduction}).

In the case of principal congruence subgroups, these results were already obtained in \cite{FLM3}.
The main new input for extending these results to arbitrary non-degenerate families is the 
estimate of \cite[\S5]{1308.3604} (cf. also \cite[\S5]{ABBGNRS}) 
for the volumes of intersections of conjugacy classes with open subgroups of $\fixco$.
Otherwise, we follow pretty much the line of argument of \cite{FLM3}.
For the spectral limit property, a key ingredient, both in \cite{FLM3} and here, is the spectral expansion obtained in \cite{MR2811597, MR2811598}.
For the geometric limit property one needs to revisit Arthur's methods and results (mostly from \cite{MR828844}) in some detail, but once again,
there is no conceptual difficulty.

While Theorem \ref{MainTheorem} and Corollary \ref{CorClassical} give a partial solution to the limit multiplicity problem (for groups $G$
satisfying \TWN\ and \BD), one may try to extend our methods to deal with the general problem.
First note that in Theorem \ref{MainTheorem} we restricted ourselves to collections $\comfam$ consisting of open subgroups of a fixed open compact subgroup
$\fixco$ of $G (\A^S)$.
It would be more natural to consider instead arbitrary open compact subgroups of $G (\A^S)$.
However, in general there is no system of representatives for the $G(\A^S)$-conjugacy classes of all such subgroups that is contained in
a compact subset of $G (\A^S)$.
Our present treatment of the geometric side follows Arthur's treatment quite closely, and consequently, the dependence
of our estimates on the support of the test function $h \otimes \one_K$ is not explicit. 
To deal with general subgroups $K$, a necessary prerequisite would be
a refinement of the main geometric estimate of Theorem \ref{thm: maingeom} that addresses this problem.

In view of the limit multiplicity problem for general lattices in groups of the form $\genG = G(F_S)$,
the following further extensions seem interesting: one could consider (assuming that $G$ satisfies the assumptions of Corollary \ref{CorClassical})
the class of all lattices in $G (F_S)$ that contain a lattice of the form $\Gamma_K$ (necessarily as a subgroup of finite index).
While any such lattice that is a subgroup of $G(F)$ is necessarily itself of the form $\Gamma_K$ (by strong approximation),
in general there also exist such lattices that are not contained in $G(F)$. The maximal lattices of this form have been described
by Prasad \cite{MR1019962} and Borel--Prasad \cite{MR1019963,MR1079647}.
A technically more demanding further step would be to treat (as in \cite{ABBGNRS})
collections of lattices $\Gamma$ in a fixed group
$\genG = G(F_S)$ that do not belong to finitely many commensurability classes. For example, one could fix a Chevalley group $G$ and an \'etale algebra
$E$ of dimension $>1$ over $\R$, and consider the lattices $G(\intgr_F)$ in the group $G(E)$, where $F$ varies over all number fields with $F\otimes\R=E$.
We note that for $G = \GL (n)$, Arthur's trace formula has been recently considered from this point of view by Jasmin Matz \cite{1310.6525, MR3314592}.

\section{The geometric limit property} \label{SectionGeom}

We start with the geometric side of the trace formula and analyze the geometric limit property, which we will prove in a refined quantitative
form in Theorem \ref{thm: maingeom} below.
The basic strategy is similar to \cite{FLM3}, where the case of principal congruence subgroups is treated. The additional ingredient
is the estimate of \cite[\S 5]{1308.3604} for the volume of the set
$\{ k \in \K^S  :  k^{-1} x k \in K \}$, where $K$ is
an open subgroup of $\fixco$ and $x \in G (\A^S)$.

\subsection{Notation}
As before, $G$ is a (connected) reductive group defined over a number field $F$ and $\A$ is the ring of adeles of $F$.
Let $\Ad: G \to \GL (\mathfrak{g})$ be the adjoint representation of $G$, whose image is the adjoint group $G^{\ad}$.
Let $S \supset S_\infty$ be a finite set of places of $F$ and $\fixco \subset G (\A^S)$ an arbitrary
compact open subgroup, which we regard as fixed in the following.

Fix a maximal $F$-split torus of $G$ with centralizer $M_0$ and fix a minimal parabolic subgroup $P_0$ of $G$ defined over $F$
containing $M_0$.
We also fix a maximal compact subgroup $\K=\prod_v \K_v = \K_\infty \K_{\fin}$ of $G(\A)=G(F_\infty)\times G(\A_{\fin})$
which is admissible with respect to $M_0$ \cite[\S 1]{MR625344}.
Let $\aaa_{M_0}^*=X^*(M_0)\otimes\R$ where $X^*(M_0)$ is the lattice of $F$-rational characters of $M_0$,
and let $\aaa_{M_0}$ be the dual space of $\aaa_{M_0}^*$.
Fix a Euclidean norm $\norm{\cdot}$ on $\aaa_{M_0}$.

As in \cite{FLM3} we fix a faithful $F$-rational representation $\rho: G \to \GL (V)$ and an $\intgr_F$-lattice $\Lambda$ in the representation space
$V$ such that the stabilizer of $\hat{\Lambda} = \hat{\intgr}_F \otimes \Lambda \subset \A_{\fin} \otimes V$ in $G(\A_{\fin})$ is the group $\K_{\fin}$.
For any non-zero ideal $\nnn$ of $\intgr_F$ let $\K(\nnn)$ be the principal congruence subgroup of level $\nnn$, i.e.,
\[
\K (\nnn) =\{ g \in G (\A_{\fin}) \, : \, \rho (g) v \equiv v \pmod{\nnn \hat{\Lambda}}, \quad \forall v \in \hat{\Lambda} \}.
\]
We denote by $\inorm (\nnn) = [ \intgr_F : \nnn]$ the ideal norm of $\nnn$.
More generally, for a finite set $S \supset S_\infty$ of places of $F$ and an ideal $\nnn\ne0$ coprime to $S_{\fin}$ let $\K^S (\nnn)=\K (\nnn) \cap \K^S$,
an open normal subgroup of $\K^S = \prod_{v \notin S} \K_v$. Similarly, for a finite set $S$ of finite places of $F$ and an ideal $\nnn\ne0$
whose prime factors are contained in $S$ let
\[
\K_S(\nnn) =\{ g \in G (F_S) \, : \, \rho (g) v \equiv v \pmod{\nnn\intgr_S\otimes\Lambda}, \quad \forall v \in \intgr_S\otimes\Lambda \}.
\]

We recall the definition of the \emph{relative level} of a compact open subgroup from \cite[\S 5.1]{FLM3}.
Let $H$ be a Zariski closed normal subgroup of $G$ defined over $F$.
We write $H(\A)^+$, or simply $H^+$, for the image of the map $H^{\SC}(\A)\rightarrow H(\A)$ where $H^{\SC}$
is the simply connected cover of the derived group of $H$. (Note that the derived group of $H$ is connected.)
For a compact open subgroup $K$ of $G(\A^S)$ we set
\begin{equation}
\level(K;H^+)=\min\{\inorm(\nnn):\K^S(\nnn)\cap H(\A^S)^+\subset K\},
\end{equation}
where $\nnn$ ranges over the ideals of $\intgr_F$ which are coprime to $S$.
We also set
\begin{equation} \label{EquationMinlevel}
\minlevel(K)=\min_H\level(K;H^+),
\end{equation}
where $H$ ranges over all Zariski closed non-central normal subgroups of $G$ defined over $F$.
(It is enough to consider the finitely many $F$-simple normal subgroups of $G$.)
Recall that a family $\comfam$ consisting of open subgroups of $\fixco$ is non-degenerate
(in the sense of Definition \ref{DefinitionNondeg} above) if and only if
the map $K \in \comfam \mapsto K \cap H (\A^S)^+$ is finite-to-one for any $F$-simple
normal subgroup $H$ of $G$. This is clearly equivalent to the condition that
$\minlevel(K)\rightarrow\infty$, $K\in\comfam$.

For each $k \ge 0$ fix a basis $\mathcal{B}_k$ of $\univ(\Lie G(F_\infty)^1 \otimes \C)_{\le k}$,
equipped with the usual filtration, and set
\[
\norm{h}_k =\sum_{X\in\mathcal{B}_k}\norm{X\star h}_{L^1(G(\A)^1)}
\]
for functions $h\in C_c^\infty(G(\A)^1)$, where we view $X$ as a left-invariant differential operator on $G (F_\infty)$.
For a compact subset $\Omega \subset G (\A)^1$
the norms $\norm{\cdot}_k$ endow the space $C^\infty_\Omega (G(\A)^1)$
of all smooth functions on $G(\A)^1$ supported in $\Omega$
with the structure of a Fr\'echet space. (It is equivalent to use
the seminorms $\sup_{x \in \Omega} \abs{(X \star h)(x)}$ for $X \in \univ(\Lie G_\infty \otimes \C)$ instead of the norms
$\norm{h}_k$, $k \ge 0$.)
Analogously, we set
\[
\norm{h}_k =\sum_{X\in\mathcal{B}_k}\norm{X\star h}_{L^1(G(F_S)^1)}
\]
for $k \ge 0$ and $h\in C_c^\infty(G(F_S)^1)$. As above, for a compact subset $\Omega_S \subset G(F_S)^1$ these norms furnish
the space $C_{\Omega_S}^\infty(G(F_S)^1)$ of all smooth functions on $G(F_S)^1$ supported in $\Omega_S$ with
the structure of a Fr\'echet space.
We also write $C_{\Omega_S} (G (F_S)^1)$ for the Banach space of continuous functions on $G(F_S)^1$ supported in $\Omega_S$.

We will use the notation $A\ll B$ to mean that there exists a constant $c$ (independent of the parameters under consideration) such that $A\le cB$.
If $c$ depends on some parameters (say $F$) and not on others then we will write $A\ll_FB$.

\subsection{The geometric side of the trace formula}
The point of departure of Arthur's trace formula is a certain distribution $J^T$ on $G (\A)^1$
which is defined for $T\in\aaa_{M_0}$ sufficiently regular in the positive Weyl chamber as the integral
over $G(F)\bs G(\A)^1$ of the so-called modified kernel (see \cite[\S6]{MR2192011}, which is based on \cite{MR518111};
see also \cite[Theorem 9.1]{MR2192011}, which is based on \cite[\S2]{MR625344}).
As a function of $T \in \aaa_{M_0}$, $J^T$ is a polynomial of degree at most
$d_0 = \dim \aaa_{M_0}-\dim X^*(G)\otimes\R$,
and the distribution $J$ (the geometric side of the trace formula) is defined
to be $J^{T_0}$, where $T_0\in\aaa_{M_0}$ is a certain distinguished point
specified in \cite[Lemma 1.1]{MR625344} that depends on $G$, $M_0$ and $\K$.
The choice of $T_0$ ensures, among other things, that $J$ does not depend on the additional choice of $P_0$,
although it still depends on $M_0$ and $\K$ (see [ibid., \S2]).

For our purposes the following characterization of the polynomials $J^T$ will be useful.
Recall the truncation function $F (\cdot, T) = F^G (\cdot, T)$ for $T \in \aaa_{M_0}$, which is the characteristic function of the truncated
Siegel domain, a certain compact subset of $G(F) \bs G (\A)^1$ (\cite[p. 941]{MR518111}, \cite[p. 1242]{MR828844}).
Also set $d(T)=\min_{\alpha\in\Delta_0}\sprod{\alpha}{T}$, where $\Delta_0$ is the set of simple roots (viewed also as linear forms on $\aaa_{M_0}$).

\begin{theorem}[Arthur] \label{thm: Arthur3.1}
Given $h\in C_c^\infty(G(\A)^1)$, $J^T(h)$ is the unique polynomial in $T \in \aaa_{M_0}$ such that
\[
\abs{J^T(h)-\int_{G(F)\bs G(\A)^1}F(x,T)\sum_{\gamma\in G(F)}h(x^{-1}\gamma x)\ dx}\ll_h (1+\norm{T})^{d_0}e^{-d(T)}
\]
for all $T$ with $d(T) \ge C$, where $C$ depends only on $h$.
\end{theorem}

This is a slight variant of \cite[Theorem 3.1]{MR828844}.
The formulation differs in two aspects. First, we state the theorem for $J^T$ as a whole and not just for the unipotent contribution.
Second, the upper bound is slightly sharper than in [loc.~cit.].
However, the proof of [loc.~cit.] is valid almost verbatim, except that
every occurrence of $\mathcal{U}_G(\Q)$ and $\mathcal{U}_1(\Q)=\mathcal{U}_{M_{P_1}}(\Q)$ has to be replaced by $G(F)$ and $M_{P_1}(F)$, respectively.
(See also the remark after \cite[Theorem 3.4]{FLM3}.)

The distribution $J^T$ (and hence $J$) can be split according to geometric conjugacy classes (see \cite{MR828844, MR835041} and below).
In particular, the contribution of a singleton conjugacy class $\{z\}$, $z\in Z(F)$, is simply the constant polynomial
$\vol (G (F) \bs G (\A)^1) h (z)$.
Write
\[
J^T_{\nc}(h)=J^T(h)-\vol (G (F) \bs G (\A)^1)\sum_{z\in Z(F)}h (z), \quad h \in C^\infty_c (G (\A)^1),
\]
and set $J_{\nc} (h) =J_{\nc}^{T_0} (h)$.
We want to estimate the latter distribution for the functions $h = h_S \otimes \one_{K}$ in terms of $K$. When $G$ is $F$-simple and simply connected,
it is possible to estimate $J_{\nc} (h_S \otimes \one_{K})$ in terms of the level of $K$. In general we have to use the modified
definition of level introduced in \eqref{EquationMinlevel} above.

\begin{theorem} \label{thm: maingeom}
There exist $\delta>0$ and an integer $k\ge0$, such that for any
compact open subgroup $\fixco$ of $G(\A^S)$, any compact subset $\Omega_S\subset G(F_S)^1$
and any integral ideal $\nnn_S$ of $\intgr_F$ whose prime factors are in $S_{\fin}$, we have
\[
J_{\nc} (h_S \otimes \one_K) \ll_{\fixco, \Omega_S} (1 + \log \inorm (\nnn_S))^{d_0}\minlevel(K)^{-\delta} \norm{h_S}_k
\]
for any bi-$\K_{S_{\fin}} (\nnn_S)$-invariant function $h_S \in C^\infty_{\Omega_S} (G (F_S)^1)$
and any open subgroup $K$ of $\fixco$.
In particular, if $\comfam$ is a non-degenerate family of open subgroups of $\fixco$, then
\[
J_{\nc} (h_S \otimes \one_K)\rightarrow0,\ \ K\in\comfam,
\]
for all $h_S \in C^\infty_c(G (F_S)^1)$.
\end{theorem}

We will prove the theorem in the rest of this section.
We end this subsection with a couple of remarks.

\begin{remark}
In the case where $G$ is anisotropic modulo the center, Theorem \ref{thm: maingeom}
(together with \S \ref{SectionStrategy})
already implies the limit multiplicity property, since the spectral
limit property is trivial in this case.
Moreover, the proof of Theorem \ref{thm: maingeom} can be much simplified, since $J_{\nc} (h)$ is then given by the
absolutely convergent integral
$\int_{G(F)\bs G(\A)^1}\sum_{\gamma\in G(F) - Z (F)}h(x^{-1}\gamma x)\ dx$.
(More general results for this case have been obtained independently in \cite{ABBGNRS}.)
\end{remark}

\begin{remark}
In the case where $K$ is a principal congruence subgroup, a quite precise quantitative estimate has already been obtained in \cite[Proposition 3.8]{FLM3}.
In the general case, as will be explained below, the constant $\delta > 0$ in Theorem \ref{thm: maingeom}
is directly derived from the estimates of \cite{1308.3604}, which are based on the approximation principle for open subgroups of $\fixco$. In particular,
any effective version of these estimates would provide an effective value of $\delta$.
\end{remark}

\subsection{Decomposition according to normal subgroups}
In Theorem \ref{thm: Arthur3.1}, $J^T(h)$ is approximated by an integral of a sum over $\gamma \in G(F)$, and
$J^T_{\nc} (h)$ is obtained by restricting the sum to non-central elements $\gamma$.
In order to deal with the case where $G$ is not necessarily $F$-simple modulo the center,
we partition the set $G(F)$ according to the Zariski closed normal subgroup of $G^{\ad}$ generated by $\Ad(\gamma)$
and decompose the distributions $J^T$ and $J_{\nc}^T$ accordingly. Equivalently, we consider (not necessarily connected) Zariski closed
normal subgroups $H$ of $G$ containing $Z_G$.

For any $G(F)$-conjugation invariant subset $C\subset G(F)$ let $\truncJ^T_C$ be the Radon measure
on $G(\A)$ (or on any closed $G(\A)$-conjugation invariant subset of $G(\A)$ containing $C$) given by
\[
\truncJ^T_C(f)=\int_{G(F)\bs G(\A)^1}F(x,T)\sum_{\gamma\in C}f(x^{-1}\gamma x)\ dx.
\]
Clearly, $\truncJ^T_{C_1\cup C_2}=\truncJ^T_{C_1}+\truncJ^T_{C_2}$ if $C_1$ and $C_2$ are disjoint.
We then have the following technical refinement of Theorem \ref{thm: Arthur3.1}.

\begin{lemma} \label{lem: Art3.1gen}
Let $H$ be a Zariski closed (not necessarily connected) normal subgroup of $G$ defined over $F$.
Then for any $f \in C_c^\infty(H(\A)^1)$ there exists a polynomial $J^T_H(f)$ in $T\in\aaa_{M_0}$ characterized by the following property.
There exist integers $k,m\ge1$ such that for any compact set $\Omega \subset H (\A)^1$,
any non-zero ideal $\nnn$ of $\intgr_F$ and any function $f \in C^\infty_\Omega(H(\A)^1)$ which is bi-invariant under $\K(\nnn)\cap H(\A)^+$, we have
\begin{equation} \label{eq: modH}
J^T_H (f) - \truncJ_{H (F)}^T(f) \ll_{\Omega}\norm{f}_k \inorm (\nnn)^m (1+\norm{T})^{d_0}e^{-d(T)},\ \ d(T) \ge d_\Omega,
\end{equation}
where $d_\Omega>0$ and the implied constant depend only on $\Omega$.
\end{lemma}

We remark that in general the distributions $\truncJ_{H (F)}^T$ and $J^T_H$ are \emph{not} simply the distributions
$\truncJ^T$ and $J^T$ with respect to $H$, even if $H$ contains the center of $G$ (which will be the only
case relevant for us). Thus, Lemma \ref{lem: Art3.1gen} is not a formal consequence of Theorem \ref{thm: Arthur3.1}
as stated. However, the proof of \cite[Theorem 3.1]{MR828844} carries over to the case at hand with minor modifications.
We defer the detailed proof of Lemma \ref{lem: Art3.1gen}, which is independent from the rest of this section, to \S \ref{ProofLemma} below.

In the following,
we will view the distributions $J^T_H$ also as distributions on $C_c^\infty(G(\A)^1)$ (by restriction of functions).

Let now $\fctrz(G)$ be the set of all Zariski closed normal subgroups $H$ of $G$ defined over $F$ such that $Z_{H}=Z_G$.
The map $H \mapsto H^{\ad}$ gives rise to a one-to-one correspondence between $\fctrz(G)$ and the set of all Zariski closed
normal subgroups of $G^{\ad}$ defined over $F$, which is just the set of all products of the $F$-simple factors of $G^{\ad}$.
For any $H \in \fctrz (G)$ set
\[
\alfctrs{H} = H - \bigcup_{H'\in\fctrz(G), \, H' \subsetneq H} H',
\]
which is a conjugation-invariant relatively open subset of $H$.
Thus, we have a partition
\[
H=\coprod_{H'\in\fctrz(G), \, H' \subset H}\alfctrs{H'}.
\]
Note that $\gamma\in G(F)$ belongs to $\alfctrs{H}(F)$ if and only if $H^{\ad}$ is the normal subgroup of $G^{\ad}$
generated by $\Ad(\gamma)$.
Define
\[
J^T_{\alfctrs{H}}(f)=\sum_{H'\in\fctrz(G), \, H' \subset H}(-1)^{r(H)-r(H')}J^T_{H'}(f)
\]
where $r(H)$ is the number of $F$-simple factors of $H^{\ad}$. Thus,
\[
J^T_H(f)=\sum_{H'\in\fctrz(G), \, H' \subset H}J^T_{\alfctrs{H'}}(f).
\]
In particular,
\[
J(f)=\sum_{H\in\fctrz(G)}J_{\alfctrs{H}}(f)
\]
and
\begin{equation} \label{eq: Jnc}
J_{\nc}(f)=\sum_{H\in\fctrz(G): \, H\ne Z_G}J_{\alfctrs{H}}(f).
\end{equation}

From Lemma \ref{lem: Art3.1gen}, applied to all $H \in \fctrz(G)$, we deduce

\begin{corollary} \label{cor: uprbnd}
There exist integers $k,m\ge1$ such that
for any compact set $\Omega \subset G (\A)^1$, any non-zero ideal $\nnn$ of $\intgr_F$ 
and any function $h \in C^\infty_\Omega(G(\A)^1)$ which is bi-invariant under the group $\K(\nnn)\cap H(\A)^+$, we have
\[
J^T_{\alfctrs{H}} (h) - \truncJ_{\alfctrs{H}(F)}^T(h) \ll_{\Omega}\norm{h}_k \inorm (\nnn)^m (1+\norm{T})^{d_0}e^{-d(T)}, \quad d(T) \ge d_\Omega,
\]
where $d_\Omega>0$ and the implied constant depend only on $\Omega$.
\end{corollary}

\subsection{An estimate for truncated integrals}
Using Corollary \ref{cor: uprbnd}, Theorem \ref{thm: maingeom} reduces to a suitable upper bound for
the truncated integrals $\truncJ_{\alfctrs{H}(F)}^T(h_S \otimes \one_K)$.
Such a bound is provided by the following lemma, which extends the result of \cite[Lemma 3.5]{FLM3} for principal congruence subgroups.

\begin{lemma} \label{truncatedlemma}
There exists $\delta>0$ such that for any $H \in \fctrz(G)$, $H \neq Z_G$, any compact open subgroup $\fixco$ of $G (\A^S)$ and any compact subset $\Omega_S \subset G (F_S)^1$ we have
\begin{equation} \label{truncatedintegral}
\truncJ^T_{\alfctrs{H}(F)}(h_S \otimes \one_K)\ll_{\fixco, \, \Omega_S} \level(K;H^+)^{-\delta}\cdot(1 + \norm{T})^{d_0}\cdot\abs{h_S}_\infty
\end{equation}
for any $h_S\in C_{\Omega_S}(G(F_S)^1)$, any $T$ such that $d(T) \ge d_{\Omega_S\fixco}$, and any open subgroup $K$ of $\fixco$.
\end{lemma}

Before proving the lemma we quickly explain how it implies Theorem \ref{thm: maingeom}.

\begin{proof}[Proof of Theorem \ref{thm: maingeom}] Let $H \in \fctrz(G)$, $H \neq Z_G$.
Arthur's interpolation argument (cf.~the proof of \cite[Proposition 3.1]{FLM3}, which is modeled after the proof of \cite[Theorem 4.2]{MR828844}),
combined with Corollary \ref{cor: uprbnd} and Lemma \ref{truncatedlemma} implies that for any compact subset $\Omega_S\subset G(F_S)^1$ we have
\[
J_{\alfctrs{H}}(h_S \otimes \one_K) \ll_{\fixco, \Omega_S} \level(K;H^+)^{-\delta}(1+\log\level(K;H^+) + \log \inorm (\nnn_S))^{d_0} \norm{h_S}_k
\]
for all bi-$\K_{S_{\fin}} (\nnn_S)$-invariant functions $h_S\in C^\infty_{\Omega_S} (G (F_S)^1)$ (where $k$ is as in
Corollary \ref{cor: uprbnd} and $\delta$ as in Lemma \ref{truncatedlemma}).
Theorem \ref{thm: maingeom} (with any smaller value of $\delta$) now follows from \eqref{eq: Jnc}.
\end{proof}

We now turn to the proof of Lemma \ref{truncatedlemma}.
We first quote the main result of \cite[\S 5]{1308.3604}.
We recall [ibid., Definition 5.2] the definition of the functions $\lambda_p^H$, where $H \in \fctrz (G)$, $H \neq Z_G$, and $p$ is a prime.
Fix a $\Z$-lattice $\Lambda_0$ in $\Lieg$ (the Lie algebra of $G$)
such that $\Lambda_0 \otimes \hat{\Z}$ is stable under the adjoint action of $\K_0 = \K_{S_{\fin}} \fixco$.
For $x \in G (\A_{\fin})$ set
\[
\lambda^H_p (x) = \max \{ n \in \Z \cup \{ \infty \} \, : \, (\Ad (x_p) - 1) \Pro_{\Lieh'} (\Lambda \otimes \Z_p) \subset p^n (\Lambda \otimes \Z_p)
\text{ for some $\Lieh' \neq 0$ } \},
\]
where $\Lieh'$ ranges over the non-trivial semisimple ideals
of the $\Q_p$-Lie algebra $\Lieh \otimes \Q_p \subset \Lieg \otimes \Q_p$,
and $\Pro_{\Lieh'}$ denotes the corresponding projection $\Lieg \otimes \Q_p \to \Lieh' \subset \Lieg \otimes \Q_p$.
(It is enough to take the simple ideals.)

We also set
\[
\Lambda^H(g)= \prod_{p: \, \lambda_p^H(g) \ge 0} p^{\lambda_p^H(g)}
\]
for all $g\in G(F)-\cup_{H'\in\fctrz(G),H'\not\supset H}H'(F)$.
Note that by \cite[Lemma 5.25]{1308.3604} (applied to the projection of $\Ad(g)$ to $H^{\ad} (F)$), $\Lambda^H(g)$ is well-defined (i.e., finite) under our
restriction on $g$.

It follows from the definition that whenever $H = H_1 \cdots H_r$ with groups $H_1, \dots, H_r \in \fctrz (G)$, we have
\[
\lambda^H_p (x) = \max_{1 \le i \le r} \lambda^{H_i}_p (x)
\]
and therefore
\begin{equation} \label{EqnLambdaLcm}
\Lambda^H (x) = \lcm_{1 \le i \le r} \Lambda^{H_i} (x).
\end{equation}

\begin{lemma} \label{LemmaConjugacyClass}
There exist constants $\varepsilon, \delta>0$ such that
for any open compact subgroup $\fixco$ of $G (\A^S)$ and any compact set $\Gamma \subset G (\A)$ we have the following estimate.
For all $H \in \fctrz (G)$, $H \neq Z_G$, $x \in H (\A^S)$, $y \in \Gamma$, and all open subgroups $K \subset \fixco$
with $N = \level ( K; H^{+} ) = \prod_p p^{n_p}$
we have
\[
\vol \left( \{ k \in \K^S \, : \, k^{-1} y^{-1} x y k \in K \} \right) \ll_{\K^S_0,\Gamma} \prod_{p|N, \, \lambda^H_p (x) < \delta n_p} p^{- \varepsilon n_p}.
\]
\end{lemma}

\begin{proof}
We may assume without loss of generality that $S = S_\infty$.
Clearly, the volume in question is bounded by the supremum of
$\vol \left( \{ k \in \K^S \cap H (\A^S) \, : \, k^{-1} y^{-1} x y k \in K \} \right)$
over $y \in \Gamma \K^S$.
The lemma follows now from \cite[Corollary 5.8]{1308.3604}
(applied to the restriction of scalars $\Res_{F/\Q}H$ of $H$ with respect to $F/\Q$ and to $y^{-1}xy \in H (\A_{\fin})$ instead of $x$),
upon noting that there exist constants $A_p$, with $A_p=0$ for almost all $p$, such that
\[
\abs{\lambda^H_p(y^{-1}xy)-\lambda^H_p(x)}\le A_p
\]
for all primes $p$, $x\in G(\A_{\fin})$, $y\in \Gamma \K^S$.
\end{proof}

Let $U$ be the unipotent radical of a standard parabolic subgroup $P$ of $G$.
We now need two simple counting lemmas for the number of elements of $U(F)$
in compact subsets of $U (\A)$ satisfying a non-degeneracy condition as well as a divisibility condition for the function $\Lambda^H$.
They are easy consequences of \cite[Lemma 3.7]{FLM3}.

In general, for a reductive group $H$ over $F$ we denote by $A_H$ the connected component
of the identity (in the usual topology)
of the group of $\R$-points of the $\Q$-split part of the center of $\Res_{F/\Q}H$, viewed as a subgroup of $H(F_\infty)$.
For a parabolic subgroup $P \supset P_0$ with Levi part $M_P$ write $A_P = A_{M_P}$ and for $T_1 \in \aaa_{M_0}$ set
\[
A_P (T_1) = \{ a \in A_{M_P} \, : \, \alpha(a)>e^{\sprod{\alpha}{T_1}} \ \forall \alpha \in \Delta_0 \}.
\]
As in \cite[p. 941]{MR518111}, we fix once and for all a suitable vector $T_1$,
depending on $G$, $M_0$, $P_0$ and $\K$, such that $G (\A) = U_0 (\A) M_0 (\A)^1 A_{P_0} (T_1) \K$.
Recall the elementary estimate for the number of rational points in conjugates of compact subsets $\Omega \subset U (\A)$, where
$U$ is the unipotent radical of $P$ \cite[Lemma 5.1]{MR828844} (cf. also \cite[Lemma 3.7]{FLM3}):
we have
\begin{equation} \label{basicUestimate}
\abs{U(F) \cap a\Omega a^{-1}} \ll_{\Omega} \modulus_P(a), \quad a\in A_P(T_1).
\end{equation}

\begin{lemma} \label{lem: latpntcnt}
Suppose that $P=M\ltimes U$ is a standard parabolic subgroup of $H \in \fctrz(G)$, $H \neq Z_G$, and $\Omega \subset U (\A)$ a compact set.
Then for every $\epsilon>0$ we have
\begin{equation} \label{eq: ucount}
\abs{\{u\in U(F) \cap\alfctrs{H} (F) \cap a\Omega a^{-1}: Y\divides\Lambda^H(u)\}} \ll_{\epsilon,\Omega} \modulus_P(a)Y^{\epsilon-1}
\end{equation}
for any positive integer $Y$ and any $a\in A_P(T_1)$.
\end{lemma}

\begin{proof}
Suppose first that $H^{\ad}$ is $F$-simple.
Then $U(F) \cap \alfctrs{H} (F) =U(F)-\{1\}$.
Write $H^{\ad}=\Res_{E/F}H_1$, where $E$ is a finite extension of $F$, $H_1$ is defined over $E$ and absolutely simple.
Let $\mathfrak{u} = \Lie U$, so that $\exp\mathfrak{u} = U(F)$, and note that $\mathfrak{u}$ has the structure of an $E$-vector space.
Since $\Omega$ is compact, there exists
an $\intgr_E$-lattice $L$ of $\mathfrak{u}$ such that $U (F_\infty) \Omega \subset U (F_\infty) \exp (L \otimes \hat{\Z})$.
For any $0\ne X\in L$ let $\idl(X) \neq 0$ be the smallest ideal $\idl$ of $\intgr_E$ with $X\in\idl L$ and let $\tilde\Lambda(X)= N(\idl(X))$.
By \cite[Lemma 5.27]{1308.3604} there exists a positive integer $D$ such that $\Lambda^H(\exp X) \divides D \tilde\Lambda(X)$
for any $0\ne X\in L$. Therefore
\[
\{1 \ne u \in U(F) \cap a \Omega a^{-1}:Y\divides\Lambda^H(u)\} \subset
\exp ( \{0\ne X\in\mathfrak{u} \cap \Ad(a)\omega:Y_2 \divides\tilde\Lambda(X)\} )
\]
for $Y_2 = Y / \gcd (Y, D)$ and a suitable compact subset $\omega$ of $\mathfrak{u} \otimes \A$.
It follows from \cite[Lemma 3.7]{FLM3} that
\[
\abs{\{0\ne X\in\mathfrak{u} \cap \Ad(a)\omega:Y_2\divides\tilde\Lambda(X)\}} \ll_{\omega} \abs{\{I \subseteq \intgr_E: \, N(I) =Y_2\}} \modulus_P(a)Y_2^{-1}.
\]
Since for every $\epsilon>0$ we have
\[
\abs{\{I \subseteq \intgr_E: \, N(I) =Y_2 \}}\ll_\epsilon Y^\epsilon,
\]
we obtain the estimate \eqref{eq: ucount} in the $F$-simple case.

In general we can decompose $H^{\ad}$ as the direct product of its $F$-simple factors $H^{\ad}_1, \dots, H^{\ad}_r$.
Then $U = \prod_{i=1}^r U_i$ and $U (F) \cap \alfctrs{H} (F) = \prod_{i=1}^r (U_i (F) - \{ 1 \})$.
There exist compact sets $\Omega_i \subset U_i (\A)$ with $\Omega \subset \prod_{i=1}^r \Omega_i$. By \eqref{EqnLambdaLcm}, we have then
\begin{multline*}
\{u = u_1 \cdots u_r \in\alfctrs{U}(F)\cap a\Omega a^{-1}: Y\divides\Lambda^H(u)\}\subset\\
\bigcup_{(Y_1, \dots, Y_r): \, Y_1 \cdots Y_r = Y} \prod_{i=1}^r \{1\ne u_i \in U_i (F)\cap a\Omega_i a^{-1}:Y_i \divides \Lambda^{H_i}(u_i) \}.
\end{multline*}
Since the number of factorizations $(Y_1, \dots, Y_r)$ of $Y$ is $\ll_\eta Y^\eta$ for any $\eta>0$, the lemma follows from the previous case.
\end{proof}

\begin{corollary} \label{cor: ltcount}
Let $H\in\fctrz(G)$, $H \neq Z_G$, $P=M\ltimes U$ a standard parabolic subgroup of $G$ defined over $F$, $\Omega\subset U(\A)$ compact and $m\in M(F)$. Then
for any $\epsilon > 0$ we have
\begin{equation} \label{eq: bndU}
\abs{\{u\in U(F)\cap a\Omega a^{-1}: mu\in \alfctrs{H}(F), Y\divides\Lambda^H(mu) \}}\ll_{\Omega,m,\epsilon} \modulus_P(a)Y^{\epsilon - 1}
\end{equation}
for all positive integers $Y$ and $a\in A_{P} (T_1)$.
\end{corollary}

\begin{proof}
Write $P_H=P\cap H$, $M_H=M\cap H$, $U_H=U\cap H$, and note that then $P_H=M_H\ltimes U_H$.
Thus, if $m\notin H(F)$ then the estimate is trivial.
Therefore we can assume that $m\in M_H(F)$ and consider only $u \in U_H(F) \cap a \Omega a^{-1}$ in \eqref{eq: bndU}.

Let $H_1\in\fctrz(G)$ be such that $m\in\alfctrs{H_1} (F)$. Clearly, we have $H_1 \subset H$.
By \cite[Lemma 5.26]{1308.3604}, there exists a positive integer $Y_1$ with
$\Lambda^{H_1}(mu) \divides Y_1$ independently of $u\in U_H(F)\cap U_H(F_\infty) \Omega$.
In particular, \eqref{eq: bndU} clearly holds if $H_1 = H$, and we may therefore assume from now on that $H_1$ is a proper subgroup of $H$.
Factor $H$ as $H=H_1H_2$ with $H_2\in\fctrz(H)$ and $H_1\cap H_2=Z_G$, which implies that
$U_H$ is the direct product of $U_1$ and $U_2$.
Then $U_2\cap\alfctrs{H_2}$ is the set of all elements of $U_2$ which are non-trivial in every $F$-simple coordinate of $H_2$, and
therefore $mu_1u_2\in \alfctrs{H}(F)$ if and only if $u_2\in U_2 (F) \cap \alfctrs{H_2}(F)$.
Also, $\Lambda^H (m u_1 u_2)$ is by \eqref{EqnLambdaLcm} the least common multiple of $\Lambda^{H_1} (m u_1)$ and $\Lambda^{H_2} (u_2)$.
Therefore for suitable compact subsets $\Omega_i \subset U_i (\A)$
the set
\[
\{ u_1 u_2 \in U_H (F) \cap a \Omega a^{-1} : m u_1 u_2 \in \alfctrs{H} (F), \ Y \divides \Lambda^H (m u_1 u_2) \}
\]
is contained in the set
\[
\{ u_1 u_2 : u_1 \in U_1 (F) \cap a \Omega_1 a^{-1}, \ u_2 \in U_2 (F) \cap \alfctrs{H_2}(F) \cap a \Omega_2 a^{-1}, \ Y_2 \divides \Lambda^{H_2} (u_2) \}
\]
for $Y_2 = Y / \gcd (Y, Y_1)$.

The corollary follows by combining \eqref{basicUestimate} for $P_1 \subset H_1$ with Lemma \ref{lem: latpntcnt} applied to $H_2$ and its
parabolic subgroup $P_2$.
\end{proof}

\begin{proof}[Proof of Lemma \ref{truncatedlemma}]
Adapting Arthur's discussion in \cite[\S 5]{MR828844} to the current situation,
we can bound $\truncJ^T_{\alfctrs{H}(F)}(h_S \otimes \one_K)$ up to a constant
by $(1 + \norm{T})^{d_0}$ times
\begin{equation} \label{eq: art1}
\sup_{a \in A_{P_0} (T_1)} \modulus_{P_0} (a)^{-1} \sum_{\gamma \in \alfctrs{H}(F)} \phi (a^{-1} \gamma a),
\end{equation}
where
\[
\phi (x) = \sup_{y\in\Gamma}\int_{\K} \abs{(h_S \otimes \one_K)(k^{-1}y^{-1}xyk)}\ dk
\]
for a compact set $\Gamma \subset G (\A)^1$ depending only on $G$, $P_0$ and $\K$.
Furthermore, \eqref{eq: art1} is bounded by the sum over standard parabolic subgroups $P=M\ltimes U$
and $\mu\in M(F)$ of
\[
\sup_{a\in A_{P} (T_1)}\modulus_P (a)^{-1} \sum_{\nu \in U(F): \, \mu\nu\in \alfctrs{H}(F)} \phi_\mu (a^{-1} \nu a),
\]
where
\[
\phi_\mu (u) = \sup_{b \in B} \modulus_{P_0} (b)^{-1} \phi (b^{-1} \mu u b), \quad u \in U_P (\A),
\]
for a fixed compact set $B \subset A_0$.
In particular, for a given $P$, $\mu$ is confined to a finite subset of $M(F)$ that depends only on $\Omega_S$.

Fix $\mu\in M(F)$. Let $N=\level(K;H^+)$ and let $\primes(N)$ be the set of prime divisors of $N$.
Let $\primes' \subset \primes(N)$ be an arbitrary subset of $\primes (N)$ and
write $N_{\primes'}=\prod_{p\in\primes'}p^{v_p(N)}$.

For any $\nu\in U(F)$ let
\[
A(N, H, \mu\nu)=\{p\in \primes (N) :\lambda_p^H(\mu\nu) < \delta v_p(N)\} \subset \primes (N).
\]
By Lemma \ref{LemmaConjugacyClass} we have
\[
\phi_\mu(a^{-1}\nu a)\ll\abs{h_S}_\infty\cdot N_{\primes'}^{-\varepsilon} \quad \text{if} \quad A (N, H, \mu\nu) = \primes'.
\]
It follows that for a suitable compact set $\Omega \subset U (\A)$ (depending on $\Omega_S$) we have:
\begin{multline*}
\sum_{\nu \in U(F): \, \mu\nu\in \alfctrs{H}(F), \ A(N, H, \mu\nu)=\primes'} \phi_\mu (a^{-1} \nu a)\ll\\
\abs{h_S}_\infty\cdot N_{\primes'}^{-\varepsilon}\cdot\abs{\{\nu \in U(F)\cap a\Omega a^{-1}: \, \mu\nu\in \alfctrs{H}(F) ,A(N, H, \mu\nu)=\primes'\}}.
\end{multline*}
On the other hand, for $A(N, H, \mu\nu)=\primes'$ clearly
\[
(N / N_{\primes'})^\delta \le \prod_{p \notin \primes'} p^{\lceil \delta v_p (N) \rceil} \divides \Lambda^H (\mu\nu).
\]
Thus, using Corollary \ref{cor: ltcount} we have for any $0 < \delta' < \min (\varepsilon,\delta)$:
\[
\sum_{\nu \in U(F): \, \mu\nu\in \alfctrs{H}(F), \ A(N, H, \mu\nu)=\primes'} \phi_\mu (a^{-1} \nu a)
\ll_{\mu,\Omega} \modulus_P (a) \abs{h_S}_\infty N_{\primes'}^{-\varepsilon}(N/N_{\primes'})^{-\delta'} \le \modulus_P (a) \abs{h_S}_\infty N^{-\delta'}.
\]
Hence, summing over all possibilities for $\primes' \subset \primes (N)$ we obtain that
\[
\sum_{\nu \in U(F): \, \mu\nu\in \alfctrs{H}(F)} \phi_\mu (a^{-1} \nu a)\ll_{\mu,\Omega}
\abs{\{\primes':\primes'\subset\primes (N) \}} \modulus_P(a)
\abs{h_S}_\infty N^{-\delta'}
\ll\modulus_P(a) \abs{h_S}_\infty N^{- \delta''}
\]
for any $0 < \delta'' < \delta'$.
The lemma follows.
\end{proof}

\subsection{Proof of Lemma \ref{lem: Art3.1gen}} \label{ProofLemma}
To conclude the proof of Theorem \ref{thm: maingeom} it remains to prove Lemma \ref{lem: Art3.1gen}.

\begin{proof}
Let $P_0^{(H)}$ be the normalizer in $G$ of $P_0\cap H$.
Thus $P_0^{(H)}$ is the largest parabolic subgroup of $G$ such that $P_0^{(H)}\cap H=P_0\cap H$.
The map $P\mapsto P\cap H$ defines a one-to-one correspondence, preserving unipotent radicals, between the
parabolic subgroups of $G$ containing $P_0^{(H)}$ and the standard parabolic subgroups of $H$. The inverse map takes $Q$ to its normalizer in $G$.
Moreover, $P\supset P_0^{(H)}$ if and only if the radical $N_P$ of $P$ is contained in $H$.

For any $f\in C_c(H(\A)^1)$ we define the modified kernel with respect to $H$ by
\[
k_{(H)}^T(x,f)=\sum_{P:P_0^{(H)}\subset P \subset G}(-1)^{\dim\aaa_P^G}\sum_{\delta\in P(F) \bs G (F)}K_P^{(H)}(\delta x,\delta x)
\hat\tau_P(H_P(\delta x)-T), \ \ x\in G (\A),
\]
where
\[
K_P^{(H)}(x,x)=\sum_{\gamma\in H(F)\cap M_P(F)}\int_{N_P(\A)}f(x^{-1}\gamma n x)\ dn
\]
and $\aaa_P^G$, $H_P$ and $\hat\tau_P$ are as in \cite{MR518111}.
We recall that only finitely many terms (depending on the support of $f$) are non-zero in the sums above.
Note that in the case $H=G$ the function $k_{(H)}^T(x,f)$ coincides with the modified kernel $k^T(x,f)=\sum_{\cls}k_{\cls}^T(x,f)$ defined in
\cite[(6.1)]{MR2192011} (following \cite{MR518111}).

As in the case $H=G$ we claim that for any $f\in C_c^\infty(H(\A)^1)$ the integral
\[
J_H^T(f):=\int_{G(F)\bs G(\A)^1}k_{(H)}^T(x,f)\ dx
\]
is absolutely convergent and that the estimate \eqref{eq: modH} holds.
These facts are proved along the same lines as \cite[Theorem 3.1]{MR828844},
which is based on the proof of \cite[Theorem 7.1]{MR518111}.
Since the modifications are mostly straightforward we only point out the differences.
In the analysis of \cite[pp. 942-945]{MR518111} we have to take into account that we sum only over $P\supset P_0^{(H)}$.
As a result $P_1(\Q)\cap M(\Q)\cap\cls$ at the bottom of [ibid., p. 944] is to be replaced by $P_1(F)\cap M(F)\cap H(F)$
which is equal to $(M_{\tilde P_1}(F)\cap H(F))N_{\tilde P_1}(F)$ where $\tilde P_1$ is the (parabolic) subgroup generated
by $P_1$ and $P_0^{(H)}$. In the ensuing discussion $N_1^P(\Q)$ and $\mathfrak{n}_1^P(\Q)$ are to be replaced by
$N_{\tilde P_1}^P(F)$ and $\mathfrak{n}_{\tilde P_1}^P(F)$ respectively.
This results in the following changes in the proof of \cite[Theorem 3.1]{MR828844} (pp.~1245-1249):
\begin{enumerate}
\item The unipotent variety $\mathcal{U}_G(\Q)$ is replaced by $H(F)$ and $\mathcal{U}_1(\Q)$ by $M_{P_1}(F)\cap H(F)=M_{\tilde P_1}(F)\cap H(F)$.
\item The sum before formula (3.1) will now be over $\{P_1,P_2:P_1\subsetneq P_2, P_2\supset P_0^{(H)}\}$.
\item From formula (3.1) onward $\mathfrak{n}_{P_1}^{P_2}(\Q)'$ is replaced by $\mathfrak{n}_{\tilde P_1}^{P_2}(F)'$ where
$\mathfrak{n}_{\tilde P_1}^{P_2}$ is the Lie algebra of $N_{\tilde P_1}\cap M_{P_2}$ and
\[
\mathfrak{n}_{\tilde P_1}^{P_2}(F)'=\mathfrak{n}_{\tilde P_1}^{P_2}(F)-\bigcup_{\tilde P_1\subset P\subsetneq P_2}
\mathfrak{n}_{\tilde P_1}^{P}(F).
\]
\item Similarly $\mathfrak{n}_{P_1}$ is replaced by $\mathfrak{n}_{\tilde P_1}$ throughout (including in the definition of
$\Phi_m(y,Y)$).
\item On p. 1246 $\modulus_{P_1}$ is replaced by $\modulus_{\tilde P_1}$ and $\modulus_{P_0}^{P_1}(a)$ by
$\modulus_{P_0}^{\tilde P_1}(aa')=\modulus_{P_0}^{P_1}(a)\modulus_{P_1}^{\tilde P_1}(aa')$.
\item There is an extra $\modulus_{P_1}^{\tilde P_1}(aa')^{-1}$ in formulas (3.3) and (3.5).
\item In the definition of $S$ at the bottom of p. 1247 and later on $\Delta_{P_0}^{P_1}$ is replaced by $\Delta_{P_0}^{\tilde P_1}$.
\end{enumerate}
The rest of the argument of \cite[\S3]{MR828844} is still valid. The point is that the extra factor $\modulus_{P_1}^{\tilde P_1}(aa')^{-1}$
compensates for the fact that the integers $k_\alpha$ on p.~1248 are only guaranteed to be positive for $\alpha\not\in\Delta_{P_0}^{\tilde P_1}$.
Moreover, the definition of $N(f)$ on p.~1247 (in the slightly modified setup) makes it clear that it suffices to assume that
$f \in C^\infty_\Omega(H(\A)^1)$ is bi-invariant under $\K(\nnn)\cap H(\A)^+$ for \eqref{eq: modH} to hold.

Finally, the argument of \cite[\S9]{MR2192011} (or \cite[\S2]{MR625344}) applies without change (except for replacing $P_0$ by $P_0^{(H)}$)
to show that $J_{H}^T(f)$ is a polynomial in $T$ for $d(T) \ge d_\Omega$.
\end{proof}

\section{The spectral limit property} \label{sec: spectral limit property}

We now turn to the spectral side of the trace formula, establish the spectral limit property
and finish the proof of Theorem \ref{MainTheorem}. For this, we use again the
group-theoretic estimates of \cite[\S 5]{1308.3604}, and combine them with the strategy of \cite[\S 7]{FLM3}. As in [ibid.], the proof proceeds in two stages.
In the first stage we prove for each proper Levi subgroup $M$ of $G$
a different property of the family of all open subgroups of a compact open subgroup of $M(\A_{\fin})$.
This intermediate statement is proved by induction over the semisimple rank of $G$. 
It is then used to derive the spectral limit property for the group $G$.
We first introduce the necessary notation and recall the appropriate inductive property.

\subsection{Polynomially bounded collections of measures} \label{SectionPB}
The technical concept of polynomial boundedness was introduced in \cite{FLM3}, following the work of Delorme \cite{MR860667}.
To recall this notion, we first introduce some notation.

Let $\theta$ be the Cartan involution of $G (F_\infty)$ defining $\K_\infty$.
It induces a Cartan decomposition $\mathfrak{g}_\infty = \Lie G (F_\infty) = \mathfrak{p} \oplus \mathfrak{k}$ with $\mathfrak{k} = \Lie \K_\infty$.
We fix an invariant bilinear form $B$ on $\mathfrak{g}_\infty$ which is positive definite on $\mathfrak{p}$ and negative definite on $\mathfrak{k}$.
This choice defines a Casimir operator $\Omega$ on $G(F_\infty)$,
and we denote the Casimir eigenvalue of any $\pi \in \Pi (G (F_\infty))$ by $\lambda_\pi$.
Fix a maximal abelian subalgebra $\aaa$ of $\mathfrak{p}\cap\Lie G(F_\infty)^1$ and let $\norm{\cdot}$ be the norm on $\aaa$ induced by $B$.
For any $r > 0$ let $\Hecke (G (F_\infty)^1)_r$ be the subspace of $\Hecke (G (F_\infty)^1)$ consisting of all functions supported
in the compact subset $\K_\infty \exp (\{ x \in \aaa\, : \, \norm{x} \le r \}) \K_\infty$ of $G (F_\infty)^1$.
For any finite set $\types \subset \Pi (\K_\infty)$ we let $\Hecke (G (F_\infty)^1)_{\types}$
be the subspace of $\Hecke (G (F_\infty)^1)$ consisting of functions whose $\K_\infty\times\K_\infty$-types
are contained in $\types\times\types$. Let also $\Hecke (G (F_\infty)^1)_{r,\types}=
\Hecke (G (F_\infty)^1)_r\cap\Hecke (G (F_\infty)^1)_{\types}$.

We recall \cite[Definition 6.2]{FLM3} that a collection $\mathfrak{M}$ of Borel measures on $\Pi (G (F_\infty)^1)$ is called \emph{\PB} if
for any finite set $\mathcal{F} \subset \Pi (\K_\infty)$ the supremum
$\sup_{\nu \in \mathfrak{M}} \abs{\nu(\hat{f})}$ is a continuous seminorm
on $\Hecke (G (F_\infty)^1)_{r,\mathcal{F}}$.
As was shown in [ibid., Proposition 6.1] (using the Paley-Wiener theorem of \cite{MR1046496}), this property is independent of $r>0$
and moreover it is equivalent to the following condition on $\mathfrak{M}$:
for any finite set $\mathcal{F} \subset \Pi (\K_\infty)$ there exists an integer $N = N (\mathcal{F})$ such that
\begin{equation} \label{polybnd4}
\sup_{\nu \in \mathfrak{M}} \nu (g_{N,\mathcal{F}}) < \infty,
\end{equation}
where $g_{N,\mathcal{F}}$ is the non-negative function on $\Pi (G (F_\infty)^1)$ defined by
\[
g_{N,\mathcal{F}} (\pi) = \begin{cases}(1+ \abs{\lambda_\pi})^{-N},&\text{if $\pi$ contains a 
$\K_\infty$-type in $\types$,}\\
0,&\text{otherwise.}\end{cases}
\]

\subsection{Review of bounds on the spectral side} \label{sec: spectralside}
We quickly recall some facts on the spectral side of the distribution $J$. It is given by
\begin{equation} \label{eq: spectral decomposition}
J (h) = \sum_{[M]} J_{\spec,M} (h), \ \ \ \ h\in C_c^\infty(G(\A)^1),
\end{equation}
with summation ranging over the conjugacy classes of Levi subgroups of $G$, represented by groups $M \supset M_0$.
The term corresponding to $M=G$ is simply $J_{\spec,G} (h) = \tr R_{\disc}(h)$.
The other terms were explicated in \cite{MR681737, MR681738} and further analyzed in \cite{MR2811597, MR2811598}.
We will not go into the (rather elaborate) description here, since for the sake of this paper all what we need is
a (conditional) estimate proved in \cite{FLM3}.
The estimate depends on properties \TWN\ and \BD\ for $G$, which were introduced in [ibid., \S5].
The former is a growth condition of the normalizing factors of global intertwining operators
while the latter is a property of the normalized local intertwining operators.
These properties are established for $G=\GL (n)$ and isogenous groups in [ibid., Proposition 5.5, Theorem 5.15],
and are conjectured to hold for any reductive group $G$.
We will not recall the precise formulation here and instead refer the reader to the discussion in \cite{FLM3}.

Let $P=M\ltimes U$ be a parabolic subgroup of $G$ defined over $F$ with $M\supset M_0$.
Let $\fctr M$ be the Zariski closed subgroup of $G$ generated by the unipotent radicals
of the parabolic subgroups of $G$ with Levi part $M$.
Thus, $\fctr M$ is defined over $F$, contained in $G^{\der}$, and normal in $G$.


Let $L^2_{\disc}(\AAA_MM(F)\bs M(\A))$ be the discrete part of $L^2(\AAA_MM(F)\bs M(\A))$, i.e., the
closure of the sum of all irreducible subrepresentations of the regular representation of $M(\A)$.
Denote by $\Pi_{\disc}(M(\A))$ the countable set of equivalence classes of irreducible unitary
representations of $M(\A)$ which occur in the decomposition of $L^2_{\disc}(\AAA_MM(F)\bs M(\A))$
into irreducibles.
For any $\pi\in\Pi_{\disc}(M(\A))$ we denote by $\AF^2_\pi(P)$ the space of automorphic functions $\varphi$ on $U(\A)M(F)\bs G(\A)$
such that for all $g\in G(\A)$ the function $m\mapsto\modulus_P(m)^{-\frac12}\varphi(mg)$ belongs to the $\pi$-isotypic component of $L^2(A_MM(F)\bs M(\A))$.
For a compact open subgroup $K$ of $G(\A_{\fin})$ and $\tau\in\Pi(\K_\infty)$ we denote by
$\AF^2_\pi(P)^{K,\tau}$ the subspace of right $K$-invariant functions which are $\tau$-isotypic with respect to $\K_\infty$.

For any finite set $S \supset S_\infty$ and any open subgroup $K_S$ of $\K_{S_{\fin}}$
let $\Hecke(G(F_S)^1)_{\types,K_S}$ be the space of all bi-$K_S$-invariant functions
$f \in \Hecke(G(F_S)^1)$ such that the function $f(\cdot g)$ belongs to $\Hecke (G (F_\infty)^1)_{\types}$ for all $g \in G(F_S)^1$.

\begin{proposition} (\cite[Lemma 7.2]{FLM3}) \label{prop: mainestimate}
Suppose that $G$ satisfies properties \TWN\ and \BD.
Let $S \supset S_\infty$ be a finite set of places of $F$ and $\types\subset\Pi (\K_\infty)$ a finite set.
Then for any sufficiently large $N>0$ there exists an integer $k \ge 0$ such that
for any open subgroup $K_S$ of $\K_{S_{\fin}}$, open compact subgroup $K$ of $G (\A^S)$ and $\epsilon>0$,
we have
\begin{multline} \label{eq: bnd33}
J_{\spec,M}(h\otimes\one_K)\\\ll_{\types,N,\epsilon}
\vol(K)\level(K_SK;\fctr M^+)^\epsilon\norm{h}_k\sum_{\substack{\tau\in\types,\\\pi\in\Pi_{\disc}(M(\A))}}
(1+\abs{\lambda_{\pi_\infty}})^{-N}\,\dim\AF^2_\pi (P)^{K_SK,\tau}
\end{multline}
for all $h\in\Hecke(G(F_S)^1)_{\types,K_S}$.
\end{proposition}


\begin{remark}
As explained in [ibid.], in practice the factor $\level(K_SK;\fctr M^+)^\epsilon$ in \eqref{eq: bnd33} can be replaced
by $(1+\log\level(K_SK;\fctr M^+))^m$, where $m$ is the co-rank of $M$ in $G$, in all cases where we can verify properties \TWN\
and \BD\ (and conjecturally in general).
However, this strengthening would not significantly improve our main result.
\end{remark}

\subsection{Integrals over Richardson orbits}
The main intermediate step in the proof of the spectral limit property is to
establish polynomial boundedness of the collection of measures $\{ \mu^{G,S_\infty}_K \}_{K \in \comfam}$
on $G (F_\infty)^1$, where $\comfam$ is the set of all open subgroups of a fixed
compact open subgroup $\K_0$ of $G (\A_{\fin})$ (Proposition \ref{BoundedLemma} below).
For the proof, we will assume by induction that this property holds for proper Levi subgroups $M$ of $G$ and explicate
Proposition \ref{prop: mainestimate} in terms of certain unipotent orbital integrals (Proposition \ref{cor: SpectralEstimate} below).

Let $d_{\ell}p$ be a left Haar measure of $P (\A_ {\fin})$ and let $\modulus_P$ be the modulus function of $P(\A_{\fin})$.
For any continuous function $f$ on $G (\A_{\fin})$ such that $f (pg) = \modulus_P (p) f(g)$ for all $p\in P(\A_{\fin})$, $g\in G(\A_{\fin})$,
the integral $\int_{P (\A_{\fin}) \backslash G (\A_{\fin})} f(g) \ dg$ is well defined
and is invariant under right translation by elements of $G (\A_{\fin})$. By our choice of measures we have
\[
\int_{P (\A_{\fin}) \backslash G (\A_{\fin})} f(g) \ dg=\int_{\K_{\fin}} f(k) \ dk.
\]
Moreover, we have
\[
\int_{G (\A_{\fin})} f(g) \ dg = \int_{P (\A_{\fin}) \backslash G (\A_{\fin})} \int_{P (\A_{\fin})} f (pg) \ d_{\ell}p \ dg
\]
for all $f \in C_c(G (\A_{\fin}))$.

For any $f \in C_c (G (\A_{\fin}))$ define
\[
\OI_{P} (f) = \int_{P (\A_{\fin}) \backslash G (\A_{\fin})}
\int_{U(\A_{\fin})} f (g^{-1}ug)\ du\ dg = \int_{\K_{\fin}}\int_{U(\A_{\fin})} f (k^{-1}uk)\ du\ dk.
\]
It is then clear that $\OI_P$ is an invariant distribution on $G (\A_{\fin})$.
(It is in fact the stable distribution corresponding to the Richardson orbit with respect to $P$.)
For any compact open subgroup $K \subset G (\A_{\fin})$ set
\[
\OI_{P,K} = \OI_{P} (\one_K).
\]


Denote by $\proj_M$ the canonical projection $P(\A_\fin)\rightarrow M(\A_\fin)$.
Note that the double coset space $P (\A_{\fin}) \backslash G (\A_{\fin}) / K$ is finite, since
$P (\A_{\fin}) \backslash G (\A_{\fin})$ is compact.

\begin{lemma} \label{LemmaRichardson}
For any compact open subgroup $K \subset G (\A_{\fin})$ we have
\begin{equation} \label{eq: OI double coset}
\OI_{P,K} = \vol (K) \sum_{\gamma \in P (\A_{\fin}) \backslash G (\A_{\fin}) / K}
(\vol_M (\proj_M (P (\A_{\fin}) \cap \gamma K \gamma^{-1})))^{-1}.
\end{equation}
\end{lemma}

\begin{proof}
We write
\begin{align*}
\OI_{P,K} & = \int_{P (\A_{\fin}) \backslash G (\A_{\fin})} \vol_U (U (\A_{\fin}) \cap g K g^{-1}) \ dg \\
& = \sum_{\gamma \in P (\A_{\fin}) \backslash G (\A_{\fin}) / K}
\int_{P (\A_{\fin}) \backslash P (\A_{\fin}) \gamma K} \vol_U (U (\A_{\fin}) \cap g K g^{-1}) \ dg.
\end{align*}
On the other hand, we claim that for each $\gamma \in G (\A_{\fin})$ we have
\begin{align*}
\vol (K) & = \int_{P (\A_{\fin}) \backslash P (\A_{\fin}) \gamma K} \vol_P (P (\A_{\fin}) \cap g K g^{-1}) \ dg \\
& = \vol_M (\proj_M (P (\A_{\fin}) \cap \gamma K \gamma^{-1}))
\int_{P (\A_{\fin}) \backslash P (\A_{\fin}) \gamma K} \vol_U (U (\A_{\fin}) \cap g K g^{-1}) \ dg,
\end{align*}
where for the second equality we use the fact that for any compact open subgroup $L \subset P (\A_{\fin})$ we can factor
\[
\vol_P (L) = \vol_M (\proj_M (L)) \vol_U (U (\A_{\fin}) \cap L).
\]
This clearly implies the lemma.

To prove the claim, we may replace $g$ by $g \gamma^{-1}$ and $K$ by $\gamma K \gamma^{-1}$ and reduce to the case $\gamma = 1$.
For $g \in P (\A_{\fin}) K$ we can write (non-uniquely) $g = p_g k_g$ with $p_g \in P (\A_{\fin})$ and $k_g \in K$.
This implies that $P (\A_{\fin}) \cap K g^{-1} = p_g^{-1} (P (\A_{\fin}) \cap p_g K g^{-1}) = p_g^{-1} (P (\A_{\fin}) \cap g K g^{-1})$.
We obtain
\begin{align*}
\vol (K) & =
\int_{P (\A_{\fin}) \backslash P (\A_{\fin}) K} \int_{P (\A_{\fin})} \one_K (pg) \ d_{\ell}p \ dg \\
& = \int_{P (\A_{\fin}) \backslash P (\A_{\fin}) K} \vol_P (P (\A_{\fin}) \cap K g^{-1}) \ dg \\
& = \int_{P (\A_{\fin}) \backslash P (\A_{\fin}) K} \vol_P (P (\A_{\fin}) \cap g K g^{-1}) \ dg, \\
\end{align*}
which establishes the claim and finishes the proof.
\end{proof}

\begin{proposition} \label{cor: SpectralEstimate}
Suppose that $G$ satisfies properties \TWN\ and \BD. Let $S \supset S_\infty$ be a finite set of places of $F$
and $\fixco$ a compact open subgroup of $G (\A^S)$.
Let $M$ be a proper standard Levi subgroup of $G$ defined over $F$. Assume
that the collection of measures $\{ \mu^{M, S_\infty}_{K_M} \}$, where
$K_M = \proj_M (P (\A_{\fin}) \cap \gamma \tilde K \gamma^{-1})$, $\tilde K$ is an open subgroup
of $\K_0 = \K_{S_{\fin}} \fixco$, and $\gamma \in G (\A_{\fin})$, is \PB.

Then for any finite set $\types\subset\Pi (\K_\infty)$ there exists an integer $k \ge 0$ such that
for any open subgroup $K_S \subset \K_{S_{\fin}}$ and $\epsilon>0$ we have
\[
J_{\spec,M}(h\otimes\one_K)\ll_{\types,\epsilon}\vol(K_S)^{-1}\norm{h}_k\cdot\OI_{P,K_SK}\cdot\level(K_SK;\fctr M^+)^\epsilon
\]
for all $h\in\Hecke(G(F_S)^1)_{\types,K_S}$ and all open subgroups $K$ of $\fixco$.
\end{proposition}

\begin{proof}
For an open subgroup $K$ of $\fixco$ write $\tilde K=K_SK$. Let $\types_M \subset \Pi (\K_{M,\infty})$ be the finite set of all irreducible components of
the restrictions of elements of $\types$ to $\K_{M,\infty}$.
Then by Frobenius reciprocity only those $\pi \in \Pi_{\disc}(M(\A))$
such that $\pi_\infty$ contains a $\K_{M,\infty}$-type in $\types_M$ can contribute to the right-hand side of \eqref{eq: bnd33}.
We denote the corresponding subset of $\Pi_{\disc}(M(\A))$ by $\Pi_{\disc}(M(\A))^{\types_M}$.

Consider now the dimensions of the spaces of automorphic forms appearing in \eqref{eq: bnd33}. We have
\begin{align*}
\dim\AF^2_\pi (P)^{\tilde K,\tau} & = m_\pi \dim \Ind^{G(\A)}_{P(\A)}(\pi)^{\tilde K,\tau} \\
& =m_\pi \dim(\Ind^{G(F_\infty)}_{P(F_\infty)}\pi_\infty)^{\tau} \dim(\Ind^{G(\A_{\fin})}_{P (\A_{\fin})}\pi_{\fin})^{\tilde K},
\end{align*}
where
\[
m_\pi=\dim\Hom(\pi,L^2_{\disc}(A_MM(F)\bs M(\A))).
\]
The factor $\dim(\Ind^{G(F_\infty)}_{P(F_\infty)}\pi_\infty)^{\tau}$ is bounded by $(\dim \tau)^2$, while
\[
\dim(\Ind^{G(\A_{\fin})}_{P (\A_{\fin})}\pi_{\fin})^{\tilde K}
=\sum_{\gamma\in P(\A_\fin)\bs G(\A_{\fin}) /\tilde K}\dim\pi_\fin^{\tilde K^\gamma_M}
\]
with
\[
\tilde K^\gamma_M = \proj_M (P (\A_{\fin}) \cap \gamma \tilde K \gamma^{-1}).
\]
Putting things together, for any $N$ there exists $k$ such that
\begin{multline*}
J_{\spec,M}(h\otimes\one_K)\ll_{\types,\epsilon}\\ \vol(K)\norm{h}_k\level(\tilde K;\fctr M^+)^{\epsilon}
\sum_{\gamma\in P(\A_\fin)\bs G(\A_{\fin}) /\tilde K} \sum_{\pi\in\Pi_{\disc}(M(\A))^{\types_M}}
(1 + \abs{\lambda_{\pi_\infty}})^{-N}m_\pi\dim\pi_\fin^{\tilde K^\gamma_M}.
\end{multline*}

By assumption, the collection
of measures $\{ \mu^{M, S_\infty}_{\tilde K^\gamma_M} \}$ is \PB. Using \eqref{polybnd4}, this means that there exists an integer $N$, depending only on $\types_M$,
such that
\[
\vol_M (\tilde K^\gamma_M) \sum_{\pi\in\Pi_{\disc}(M(\A))^{\types_M}}
(1 + \abs{\lambda_{\pi_\infty}})^{-N}m_\pi\dim\pi_\fin^{\tilde K^\gamma_M} \ll_{\types_M} 1
\]
as $K \subset K_0^S$ and $\gamma \in G (\A_{\fin})$.
By Lemma \ref{LemmaRichardson}, we obtain from this that
\[
\vol(\tilde K)\sum_{\gamma\in P(\A_\fin)\bs G(\A_{\fin}) /\tilde K} \sum_{\pi\in\Pi_{\disc}(M(\A))^{\types_M}}
(1 + \abs{\lambda_{\pi_\infty}})^{-N}m_\pi\dim\pi_\fin^{\tilde K^\gamma_M}\ll_{\types_M}\OI_{P,\tilde K}.
\]
The proposition follows.
\end{proof}

We remark that the assumption that $\tilde K$ is a subgroup of $\K_{S_{\fin}} \fixco$ 
did not play any role in the proof.
Thus, the same argument yields the following more general result (which will not be used in the remainder of this paper).

\begin{proposition}
Suppose that $G$ satisfies properties \TWN\ and \BD. Let $S \supset S_\infty$ be a finite set of places of $F$
and $\comfam^S$ a collection of compact open subgroups of $G (\A^S)$.
Let $M$ be a proper standard Levi subgroup of $G$ defined over $F$. Assume
that the collection of measures $\{ \mu^{M, S_\infty}_{K_M} \}$, where
$K_M = \proj_M (P (\A_{\fin}) \cap \gamma K_S K^S \gamma^{-1})$, $K_S$ is
an open subgroup of $\K_{S_{\fin}}$, $K^S \in \comfam^S$, and $\gamma \in G (\A_{\fin})$, is \PB.

Then for any finite set $\types\subset\Pi (\K_\infty)$
there exists an integer $k \ge 0$ such that
for any open subgroup $K_S \subset \K_{S_{\fin}}$ and $\epsilon>0$ we have
\[
J_{\spec,M}(h\otimes\one_K)\ll_{\types,\epsilon}\vol(K_S)^{-1}\norm{h}_k\cdot\OI_{P,K_SK}\cdot\level(K_SK;\fctr M^+)^\epsilon
\]
for all $h\in\Hecke(G(F_S)^1)_{\types,K_S}$ and all $K \in \comfam^S$.
\end{proposition}

\subsection{Completion of the proof}
We can now prove polynomial boundedness by induction over the Levi subgroups of $G$. The group-theoretic ingredient is
the following estimate for $\OI_{P,K}$ from \cite{1308.3604}.

\begin{lemma}(\cite[Corollary 5.28]{1308.3604}) \label{lem: mainspecorbitbnd}
There exists $\delta>0$ such that for all compact open subgroups $\K_0 \subset G (\A_{\fin})$ we have
\[
\OI_{P,K} \ll_{\K_0} \level(K;\fctr M^+)^{-\delta},
\]
as $K$ ranges over the open subgroups of $\K_0$.
\end{lemma}

From this we obtain our main technical result.

\begin{proposition} \label{BoundedLemma}
Suppose that $G$ satisfies \TWN\ and \BD.
Then for any compact open subgroup $\K_0$ of $G(\A_{\fin})$ the collection of measures $\{ \mu^{G, S_\infty}_{K} \}$,
$K$ ranging over the open subgroups of $\K_0$, is \PB.
\end{proposition}

\begin{proof}
First note that it is enough to prove the statement for $\K_0=\K_{\fin}$.
This follows from the elementary inequality
\[
\mu_{K_1}\le [K_1:K_2]\mu_{K_2},
\]
which holds for any compact open subgroups $K_1\supset K_2$ of $G(\A_{\fin})$.

We prove the statement for $\K_0=\K_{\fin}$ by induction on the semisimple rank of $G$.
The base of the induction is \cite[Lemma 7.6]{FLM3}.
We recall that properties \TWN\ and \BD\ are hereditary for Levi subgroups.
For the induction step, we can therefore assume that for any proper Levi subgroup $M$ of $G$ the collection of measures
$\{ \mu^{M, S_\infty}_{K} \}$, $K$ ranging over the open subgroups of $\K_{M,\fin} = \K \cap M (\A_{\fin})$, is \PB.

Fix $r > 0$ and apply the trace formula to $h \otimes \one_{K}$, where $h \in \Hecke (G(F_\infty)^1)_{r,\types}$.
By \eqref{eq: spectral decomposition} we have
\[
\vol(G(F)\bs G(\A)^1) \mu^{G, S_\infty}_{K} (\hat{h}) = J (h \otimes \one_{K}) - \sum_{[M], \, M \neq G} J_{\spec,M} (h \otimes \one_{K}).
\]
By Theorem \ref{thm: maingeom} (with $S=S_\infty$), $\sup_K \abs{J (h \otimes \one_{K})}$ 
is a continuous seminorm on the space $\Hecke (G(F_\infty)^1)_{r,\types}$.
On the other hand, applying Proposition \ref{cor: SpectralEstimate} with $S=S_\infty$ (taking into account the induction hypothesis
and the fact that $G(\A_{\fin})=P(\A_{\fin})\K_{\fin}$) and Lemma \ref{lem: mainspecorbitbnd},
we infer that every spectral term $\sup_K\abs{J_{\spec,M} (h \otimes \one_{K})}$ is also a continuous seminorm on $\Hecke (G(F_\infty)^1)_{r,\types}$.
We conclude that the collection $\{ \mu^{G, S_\infty}_{K}\}$ is \PB.
\end{proof}

\begin{remark} \label{RemarkIndependencePB}
We expect that even the collection $\{ \mu^{G,S_\infty}_K \}_{K \in \comfam}$, where $\comfam$ is the set of \emph{all} compact open subgroups
of $G (\A_{\fin})$, is polynomially bounded. As already mentioned in the introduction, for general $G$ the estimates on the geometric and spectral sides
contained in this paper are not sufficient to show this. However, for the groups $G = \GL (n)$ and $\SL (n)$ it is easy to see that the general boundedness statement
is true, since for $G=\GL(n)$ all maximal compact subgroups of $G (\A_{\fin})$ are conjugate, and for $G=\SL(n)$ they fall into finitely many classes under
the action of $\GL (n, F) \SL (n, \A_{\fin})$.
\end{remark}

As before, let $S$ be a finite set of places containing $S_\infty$.

\begin{corollary}[Spectral limit property] \label{corspectrallimit}
Suppose that $G$ satisfies properties \TWN\ and \BD.
Let $\fixco$ be a compact open subgroup of $G (\A^S)$.
Then the spectral limit property holds with respect to any non-degenerate family $\comfam$ of open subgroups of $\fixco$.
\end{corollary}

\begin{proof}
By Proposition \ref{BoundedLemma}, for every Levi subgroup $M$ and any fixed compact open subgroup $\fixco \subset G (\A^S)$ the collection of the
measures $\{ \mu^{M, S_\infty}_{K_M} \}$ for $K_M = \proj_M (P (\A_{\fin}) \cap \gamma K \gamma^{-1})$, $K$ ranging over the open subgroups
of $\K_{S_{\fin}} \fixco$, $\gamma \in G (\A_{\fin})$, is \PB. Indeed, we may restrict here $\gamma$ to a set of representatives for the finitely
many double cosets of $P (\A_{\fin}) \backslash G (\A_{\fin}) / \K_{S_{\fin}} \fixco$.

Therefore, as in the proof of Proposition \ref{BoundedLemma}, we can apply
Proposition \ref{cor: SpectralEstimate} and Lemma \ref{lem: mainspecorbitbnd}
to conclude that for any $M \neq G$ and $h \in \Hecke (G(F_S)^1)$ we have
\[
J_{\spec,M} (h \otimes \one_{K}) \to 0
\]
for a collection $\comfam$ of open subgroups of $\fixco$,
provided that $\level(K;\fctr M^+)\rightarrow\infty$, $K \in \comfam$.
Thus, if $\minlevel(K)\rightarrow\infty$, $K \in \comfam$, then we have by \eqref{eq: spectral decomposition}:
\[
J (h \otimes \one_{K}) - \tr R_{\disc} (h \otimes \one_{K}) \to 0,
\]
which is the spectral limit property.
\end{proof}

The main result of the paper, Theorem \ref{MainTheorem}, now follows from
Theorem \ref{thm: maingeom}, Corollary \ref{corspectrallimit}, and the discussion in \S \ref{SectionStrategy}.


\newcommand{\etalchar}[1]{$^{#1}$}
\providecommand{\bysame}{\leavevmode\hbox to3em{\hrulefill}\thinspace}
\providecommand{\MR}{\relax\ifhmode\unskip\space\fi MR }
\providecommand{\MRhref}[2]{%
  \href{http://www.ams.org/mathscinet-getitem?mr=#1}{#2}
}
\providecommand{\href}[2]{#2}

\end{document}